\documentclass[paper=a4,headsepline,DIV13]{scrartcl}
\usepackage[utf8]{inputenc}

\usepackage{amsmath,amsthm,amssymb}
\usepackage{graphicx}
\usepackage{enumerate}
\usepackage{booktabs}
\usepackage{hyperref}
  
\DeclareMathOperator{\dist}{dist}
\DeclareMathOperator{\diam}{diam}
\DeclareMathOperator{\supp}{supp}

\newtheorem{lemma}{Lemma}
\newtheorem{theorem}{Theorem}
\newtheorem{corollary}{Corollary}
\newtheorem{remark}{Remark}

\usepackage{todonotes}

\title{Error estimates for variational normal derivatives and Dirichlet control problems with energy regularization}
\author{Max Winkler
\thanks{Technische Universit\"at Chemnitz, Faculty of Mathematics,
  Professorship Numerical Mathematics (Partial Differential Equations)}}

\begin{document}

\maketitle

\begin{abstract}
  This article deals with error estimates for the finite element approximation
  of variational normal derivatives and, as a consequence, error estimates 
  for the finite element approximation of Dirichlet boundary control problems 
  with energy regularization. The regularity of the solution is
  carefully carved out exploiting weighted Sobolev and H\"older spaces.
  This allows to derive a sharp relation between the convergence rates for the
  approximation and the structure of the geometry, more precisely, the largest opening angle at the
  vertices of polygonal domains. 
  Numerical experiments confirm that the derived convergence rates are sharp.
\end{abstract}

\section{Introduction}

The problem investigated in this article is the optimal Dirichlet control
problem
\begin{equation}\label{eq:target_funct}  
 \min_{z\in H^{1/2}(\Gamma)} \left\lbrace \frac12\,\|u(z)-u_d\|_{L^2(\Omega)}^2 +
 \frac{\nu}2\, |z|_{H^{1/2}(\Gamma)}^2\right\rbrace,
\end{equation}
where $u(z)\in H^1(\Omega)$ is the solution of the boundary value problem
\begin{equation}\label{eq:state}
-\Delta u = 0 \ \mbox{in}\ \Omega,\qquad u=z\ \mbox{on}\ \Gamma.
\end{equation}
The domain $\Omega\in\mathbb R^2$ is assumed to have a polygonal boundary
$\Gamma$. The function $u_d\in L^2(\Omega)$ is referred to as desired state.
The parameter $\nu>0$ is a regularization parameter and the
corresponding term in the objective guarantees
the existence of a solution in the space $H^{1/2}(\Gamma)$.

This optimal control problem has first been formulated by Lions \cite{Lio71}. Later, a
regularization using the $L^2(\Gamma)$-norm of the control became more
attention \cite{CR06,MRV13,AMPR16}. From the modeling point of view, the
$L^2(\Gamma)$ regularization is
reasonable as the regularization term can be interpreted as a measure for
control costs, but the disadvantage is that the control has a rather unexpected behavior
near the corners. In the general case the control tends to $0$ at convex and to
infinity at reentrant corners \cite{AMPR15}. Thus, the idea of using an energy regularization
instead was revealed by Of et.\ al.\ \cite{OPS13}. It has to be noted that the
behavior near the corners is in this approach 
just shifted to the tangential derivatives of the control. The physical
interpretation of the regularization term using the $H^{1/2}(\Gamma)$-norm of
the control is, that it is 
equivalent to the energy norm of the corresponding state $u(z)$, which might
be, depending on the
concrete application, a measure for control costs as well.
This becomes clear when defining the seminorm in $H^{1/2}(\Gamma)$ by
\begin{equation*}
 |z|_{H^{1/2}(\Gamma)}^2 := \int_\Gamma \partial_n u(z)\,z = \|\nabla u(z)\|_{L^2(\Omega)}^2.
\end{equation*}
Closely related are the investigations for the Neumann control problem
with an $H^{-1/2}(\Gamma)$-regularization \cite{ApelSteinbachWinkler2016,Win15}.
Note, that the optimal state is in both approaches equivalent.

Error estimates for approximate solutions of the Dirichlet control problem
are discussed already in 
\cite{OPS13} where all variables are approximated by piecewise linear finite
elements.
For this approach, and in case of convex computational domains, 
the convergence rate of $1$ for the control in the
$H^{1/2}(\Gamma)$-norm was proved, but in the numerical experiments a higher
convergence rate is observed. The results in the present article will show
that the rate $1$ is only a worst-case estimate for convex domains, meaning,
that if an opening angle of a corner tends to $180^\circ$, the convergence
rate will tend to $1$. The same convergence rate is proved in \cite{JSW18} 
for arbitrary polygonal domains for a discretization using the energy
corrected finite element method.

It is the aim of the present paper to prove sharp convergence rates.
Depending on the opening angle at
the corners one can prove a convergence rate up to 
$3/2$ for the control in the
$H^{1/2}(\Gamma)$-norm.
It turns out that this is in general only possible when the opening angles are
all less than $120^\circ$ as the corresponding singularities are mild enough
to guarantee $H^2(\Gamma)$-regularity of the control. 

The difficult part of the convergence proof is to derive an error estimate for a
variational normal derivative of the finite element solution of
the Poisson and the Laplace equation in the $H^{-1/2}(\Gamma)$-norm.
Such an error term appears due do the approximation of the 
Steklov-Poincar\'e operator $z\mapsto \partial_n u(z)$ used to realize the $H^{1/2}(\Gamma)$-norm, 
and the approximation for the normal derivative of the adjoint state variable
which appears in the optimality condition. A worst-case estimate for
variational normal derivatives in the $H^{-1/2}(\Gamma)$-norm, as used in
\cite{OPS13}, can be easily derived when using a trace theorem and standard 
finite element error estimates. Sharp error estimates require some more
effort and will be discussed intensively in the present article.
Closely related are the error estimates in the $L^2(\Gamma)$-norm
for the exact normal derivative of the finite element
approximation from \cite{HMW14,PW17}. In the latter reference the variational normal derivative
used in the present paper is discussed as well.
In the present article we consider estimates for the variational normal derivative
in $H^{-1/2}(\Gamma)$.
The convergence rate we prove will be related to $\omega_{max}$ denoting the
largest opening angle of the corners of the domain $\Omega$. Moreover, $y$ and $y_h$
are the solution of the Poisson or Laplace equation and its finite element approximation,
respectively. Under the assumption that the input data are sufficiently smooth,
and the normal derivative is continuous in the corners when a convergence
rate larger than $1$ is expected,
we show that the variational normal derivative satisfies the estimate
\begin{equation*}
 \|\partial_n y - \partial_n^h y_h\|_{H^{-1/2}(\Gamma)} \le c\,
 h^{\min\{3/2,\pi/\omega_{max}-\varepsilon\}}
\end{equation*}
with a constant $c>0$ independent of the mesh size $h$, and arbitrary but
sufficiently small $\varepsilon>0$.
The proof is based on an idea developed in \cite{PW17} where estimates
in the $L^2(\Gamma)$-norm on a sequence of 
boundary concentrated meshes is proved. 

As an application, we use this result to derive sharp discretization
error estimates for the optimal control problem \eqref{eq:target_funct}--\eqref{eq:state}.
Therefore, we approximate the control, state and adjoint state by a linear finite
element discretization.
Under the assumption that $u_d$ is H\"older continuous in case of convex
$\Omega$, or belongs to $L^2(\Omega)$ in case of non-convex $\Omega$, 
we show the same convergence rate for the
control approximation in the $H^{1/2}(\Gamma)$-norm, this is,
\begin{equation*}
 \|z - z_h\|_{H^{1/2}(\Gamma)} \le c\,
 h^{\min\{3/2,\pi/\omega_{max}-\varepsilon\}},
\end{equation*}
where $z$ and $z_h$ are the continuous and discrete optimal control.
This confirms the behavior figured out in the 
numerical experiments from \cite{OPS13} on the unit square,
where the rate $3/2$ was predicted numerically. The conjecture that this rate
is achieved on arbitrary convex polygonal domains is obviously wrong. Our
theory promises that this rate is obtained unless all opening angles of
corners are less that $2\pi/3$ which is also confirmed by numerical
experiments. The worst-case convergence rate of $1$ is indeed achieved unless
the domain remains convex. If the largest angle tends to $2\pi$, the
convergence rate will tend to~$1/2$.

As a further application of estimates for variational normal derivatives we mention
Steklov-Poincar\'e operators that are frequently used for parallel finite element methods
relying on domain decomposition \cite{AL85,QV99,XZ97}.
Closely related are the error estimates from \cite{MW12}. Therein, the authors
derive optimal error estimates for discrete Lagrange multipliers in $H^{-1/2}(\Gamma)$
defined on the interfaces of the subdomains. 
The approximation of the multipliers corresponds to some variational
approximation of a normal derivative as well.

The article is structured as follows. In Section~\ref{sec:regularity} 
we collect a priori estimates for solutions of the Poisson and Laplace
equation in weighted norms involving a regularized boundary distance
function. Moreover, we have to carve out the singular behavior near corners
of the domain which is done by weighted Sobolev and H\"older spaces. To this end, we
provide the required shift theorems.
Error estimates for the solution of the Dirichlet boundary value problem in the $L^2(\Omega)$-
and $H^1(\Omega)$-norm as well as for the discrete normal derivatives in
the $H^{-1/2}(\Gamma)$-norm are derived
in Section~\ref{sec:bvp}. These estimates are applied to the discretization of
our optimal control problem in Section~\ref{sec:control}. The results
derived therein are confirmed by the numerical experiments in Section~\ref{sec:experiments}.

\section{Auxiliary results}\label{sec:regularity}

Let us first explain the notation we will use in this paper.
The computational domain is denoted by $\Omega\subset\mathbb R$ and is always assumed to
have a polygonal boundary $\Gamma$.
By $W^{k,p}(\Omega)$, $k\in\mathbb N_0$, $p\in[1,\infty]$ we denote the usual Sobolev spaces
and write $H^{k}(\Omega) := W^{k,2}(\Omega)$, $L^2(\Omega):=H^0(\Omega)$. Frequently, we use the space
$H_0^1(\Omega)$ which is the closure of $C_0^\infty(\Omega)$ with respect to the $H^1(\Omega)$-norm.
For the corresponding norms and inner products we write $\|\cdot\|_{X}$ and $(\cdot,\cdot)_X$, respectively. The subscript $X$ indicates the related space.
Moreover, $\left<\cdot,\cdot\right>$ stands for the dual pairing between
$H^{-1/2}(\Gamma)$ and $H^{1/2}(\Gamma)$.

The aim of this section is to collect some regularity results for the solution of 
the Laplace and Poisson equation. The weak form
reads: Find $y\in H^1(\Omega)$ satisfying 
\begin{equation}\label{eq:bvp}
  y|_\Gamma = g,\qquad  (\nabla y,\nabla v)_{L^2(\Omega)^2} = (f,v)_{L^2(\Omega)}\qquad\forall v\in H^1_0(\Omega).
\end{equation}
The functions $f\in L^2(\Omega)$ and $g\in H^{1/2}(\Gamma)$ are given input data.

\subsection{Weighted regularity}\label{sec:weighted_est_bd}

For technical reasons we recall some a priori estimates in weighted norms involving
the weight function $\sigma(x):=\kappa h + \dist(x,\Gamma)$ with arbitrary $\kappa>0$.
This is a regularized distance function with respect to the boundary of the domain $\Omega$.
The following result is proved already in \cite[Lemma 1]{PW17}.
\begin{lemma}\label{lem:weighted_h1}
 Let $w\in H^1_0(\Omega)$ be the weak solution of $-\Delta w = f$ in $\Omega$.
 Then, the a priori estimate
 \begin{equation*}
  \|\sigma^{-1}\,w\|_{L^2(\Omega)} \le c\, \|\nabla w\|_{L^2(\Omega)} \le c\, \|\sigma\,f\|_{L^2(\Omega)}
\end{equation*}
holds.
\end{lemma}
Furthermore, we will need an interior regularity result:
\begin{lemma}\label{lem:interior_H2}
  Let $w\in H^1(\Omega)$ satisfy
  \[
    (\nabla w,\nabla v)_{L^2(\Omega)^2} = (f,v)\qquad\forall v\in H^1_0(\Omega)
  \]
  with some function $f\in L^2(\Omega)$.
  Moreover, let be given $\Omega_0\subset\subset\Omega_1\subset\Omega$ and
  denote by $d:=\dist(\partial\Omega_1,\partial\Omega_0)$ the distance
  between the boundaries of $\Omega_0$ and $\Omega_1$. Then, the estimate
 \[
  \|\nabla^2 w\|_{L^2(\Omega_0)} \le c\, \left(\|f\|_{L^2(\Omega_1)} + d^{-1}\,\|\nabla w\|_{L^2(\Omega_1)}\right)
\]
is valid.
\end{lemma}
\begin{proof}
  The estimate (i) can be concluded from the proof of \cite[Theorem 8.8]{GT98} where this assertion is stated with a generic constant depending on the quantity $d$ that we want to carve out exactly.
 Thus, we repeat the proof for the convenience of the reader. 
 The proof basically relies on \cite[Lemma 7.24]{GT98} which states that a function $u\in L^2(\Omega)$ belongs to $H^1(\Omega_0)$ if its difference quotients 
 $D^h_k u(x):= \frac1h(u(x+h\mathbf e_k) - u(x))$, $k\in\{1,2\}$, are bounded in the $L^2(\Omega_0)$-norm for all $h\in\mathbb R$ with $|h|$ sufficiently small such that $D_k^h$ is well-defined 
 in $\Omega_1$. Moreover, the inclusion
 \begin{equation}\label{eq:trick_GT}
  \|D_k^h w\|_{L^2(\Omega_0)} \le K
  \quad\Rightarrow\quad 
  \|\partial_k w\|_{L^2(\Omega_0)} \le K
 \end{equation}
 is valid.
 To conclude the desired estimate we thus have to confirm that $\|D_k^h \nabla w\|_{L^2(\Omega_0)}$ is bounded.
 For technical reasons we introduce a further set $\tilde\Omega$ satisfying $\Omega_0\subset\subset\tilde\Omega\subset\subset \Omega_1$
 and $\dist(\Omega_0,\partial\tilde\Omega)\sim d$.
 For an arbitrary test function $v\in H^1_0(\Omega)$ with $\dist(\supp v,\partial\tilde \Omega)>2h$
 we obtain
 \begin{align}\label{eq:weak_form_diff}
  \int_\Omega (\nabla D_k^h w) \cdot \nabla v &= -\int_\Omega\nabla w\cdot(\nabla D_k^{-h} v) 
  = -\int_\Omega f\, D^{-h}_k v \nonumber\\
  &\le \|f\|_{L^2(\Omega_1)}\,\|\nabla v\|_{L^2(\Omega)}.
 \end{align}
 In the last step we bounded the difference quotient by the first derivative of $v$. Such an estimate is proved in \cite[Lemma 7.23]{GT98}.
 Next, we introduce a smooth cut-off function $\eta\in C_0^\infty(\Omega)$ satisfying $\eta\equiv 1$ in $\Omega_0$ and $\supp \eta\subset\tilde\Omega$.
 Moreover, $\eta$ is constructed in such a way that $|\nabla\eta|\le c\, d^{-1}$. 
 For sufficiently small $h$ we obtain from the product rule and \eqref{eq:weak_form_diff} for $v=\eta^2 D_k^h w$
 \begin{align}\label{eq:diff_quot_estimate}
   &\|\eta\,D_k^h \nabla w\|_{L^2(\Omega)}^2
   = \int_\Omega (\nabla D_k^h w)\cdot (\eta^2 \,(\nabla D_k^h w)) \nonumber\\
   &\quad= \int_\Omega (\nabla D_k^h  w)\cdot(\nabla (\eta^2\, D_k^h w) - 2\,\eta\, \nabla \eta D_k^h w) \nonumber\\
   &\quad \le \|f\|_{L^2(\Omega_1)}\,\|\nabla (\eta^2\,D_k^h w)\|_{L^2(\Omega)}
     + c\,d^{-1}\,\|\eta\, D_k^h \nabla w\|_{L^2(\tilde\Omega)}\,\|D_k^h w\|_{L^2(\tilde \Omega)}.  
 \end{align}
 Again, we apply \cite[Lemma 7.23]{GT98} to obtain $\|D_k^h w\|_{L^2(\tilde \Omega)}\le c\, \|\nabla w\|_{L^2(\Omega_1)}$.
 Moreover, with the product rule we obtain
 \begin{align*}
   \|\nabla (\eta^2\,D_k^h w)\|_{L^2(\Omega)}
   &\le 2\,\|\eta\,\nabla \eta\,D_k^h w\|_{L^2(\Omega)}
     + \|\eta^2\, \nabla D_k^h w\|_{L^2(\Omega)}\\
  &\le c\, \left(d^{-1}\,\|\nabla w\|_{L^2(\Omega_1)} + \|\eta\, D_k^h \nabla w\|_{L^2(\Omega)}\right).
 \end{align*}
 Insertion of this estimate into \eqref{eq:diff_quot_estimate} yields with Young's inequality and a kick-back argument for the latter term on the right-hand side
 \begin{equation*}
   \|\nabla D_k^h w\|_{L^2(\Omega_0)} \le \|\eta\,\nabla D_k^h w\|_{L^2(\Omega)}
   \le c\,\left(\|f\|_{L^2(\Omega_1)} + d^{-1}\,\|\nabla w\|_{L^2(\Omega_1)}\right).
 \end{equation*}
 The desired estimate then follows from \eqref{eq:trick_GT}.
\end{proof}

\subsection{Weighted Sobolev and H\"older spaces}\label{sec:weighted_spaces}

In order to describe the regularity of the solution of boundary value problems
in an accurate way we exploit regularity results in weighted Sobolev spaces.
These spaces capture the corner singularities contained in the solution and
allow us to derive sharp interpolation error estimates.
Throughout the paper we denote the corners of $\Omega$ by $\boldsymbol c_j$, $j\in\mathcal C:=1,\ldots,d$. Moreover, denote by $\Gamma_j$ the boundary edge having endpoints $\boldsymbol c_j$ and $\boldsymbol c_{j+1}$ or $\boldsymbol c_1$ in case of $j=d$. The interior angle between the edges
intersecting in $\boldsymbol c_j$ is $\omega_j\in (0,2\pi)$.

In order to introduce the weighted Sobolev spaces used for the analysis,
we divide the domain into circular sectors $\Omega_R^j:=\{x\in\Omega\colon |x-\boldsymbol c_j| < R\}$, $j\in\mathcal C$, with sufficiently small $R$ such that these sectors do not overlap.
The remaining sets are denoted by $\hat\Omega_R:=\Omega\setminus\cup\{\Omega_R^j\colon j\in\mathcal C\}$.
  For each $k\in\mathbb N_0$, $p\in [1,\infty)$ and some weight $\beta\in\mathbb R_+$ ($\mathbb R_+:=[0,\infty)$), we introduce the local norms
  \begin{align*}
    \|u\|_{V^{k,p}_{\beta}(\Omega_R^j)}^p&:=\sum_{|\alpha|\le k}\|r_j^{\beta-k+|\alpha|}\,D^\alpha u\|_{L^p(\Omega_R^j)}^p,\\
      \|u\|_{W^{k,p}_{\beta}(\Omega_R^j)}^p&:=\sum_{|\alpha|\le k}\|r_j^{\beta}\, D^\alpha u\|_{L^p(\Omega_R^j)}^p,
    \end{align*}
    and for an analogous definition in case of $p=\infty$, the sum has to be replaced by the maximum over $|\alpha|\le k$.
    For some $\vec\beta\in\mathbb R_+^d$ the global norms are defined by
    \begin{equation*}
      \|u\|_{V^{k,p}_{\vec\beta}(\Omega)}:=\left(\sum_{j\in\mathcal C}\|u\|_{V^{k,p}_{\beta_j}(\Omega_R^j)}^p + \|u\|_{W^{k,p}(\hat\Omega_R/2)}^p\right)^{1/p},
    \end{equation*}
    in case of $p\in [1,\infty)$ and with the obvious modification for $p=\infty$.
    When replacing $V$ by $W$ in the definition above, we obtain the global norm $\|\cdot\|_{W^{k,p}_{\vec\beta}(\Omega)}$. 
    The weighted Sobolev spaces $V^{k,p}_{\vec\beta}(\Omega)$ and $W^{k,p}_{\vec\beta}(\Omega)$ are defined as the set of functions whose norms introduced above are finite.
    The trace spaces are denoted by $V^{k-1/p,p}_{\vec\beta}(\Gamma)$ and
    $W^{k-1/p,p}_{\vec\beta}(\Gamma)$, respectively.
    The previous definitions and an intensive discussion on the relation between $V$- and $W$-spaces
    can be found in \cite[Chapter 4,\ \S5]{NP94}, \cite[Section 6.2]{MR10}.

        Later, we will frequently derive error estimates where the convergence rate will depend on the largest weight. Thus, we define
        \[\overline{\vec\beta}:=\max_{j\in\mathcal C} \beta_j.\]
    In the next chapter, we will frequently exploit regularity results in these space with $p=\infty$,
    but for this case a shift theorem is not valid. As a remedy, weighted H\"older spaces are
    used and we take the definition from \cite[Section 6.7.1]{MR10}.
    Again, we define some local norms with parameters $k\in\mathbb N_0$, $\sigma\in (0,1]$ and $\delta\ge \sigma$, defined by
    \begin{align*}
      \|u\|_{\Lambda^{k,\sigma}_{\delta}(\Omega_R^j)}&:= \sup_{x\in \Omega_R^j}\sum_{|\alpha|\le k} r_j(x)^{\delta-k-\sigma+|\alpha|}\,|D^\alpha u(x)| + \left<u\right>_{k,\sigma,\beta,\Omega_R^j},\\
      \|u\|_{C^{k,\sigma}_{\delta}(\Omega_R^j)}&:= \sup_{x\in \Omega_R^j}\sum_{|\alpha|\le k} r_j(x)^{\max\{0,\delta-k-\sigma+|\alpha|\}}\,|D^\alpha u(x)| + \left<u\right>_{k,\sigma,\delta,\Omega_R^j},
    \end{align*}
    where the seminorm is defined by
    \begin{equation*}
      \left<u\right>_{k,\sigma,\delta,\Omega_R^j}:=\sup_{x,y\in \Omega_R^j}\sum_{|\alpha|=k}
      \frac{|r_j(x)^{\delta}\,D^\alpha u(x) - r_j(y)^{\delta}\,D^\alpha u(y)|}{|x-y|^{\sigma}}.
    \end{equation*}
    The global norm is then given by
    \begin{equation*}
      \|u\|_{\Lambda^{k,\sigma}_{\vec\delta}(\Omega)} := \sum_{j\in\mathcal C}\|u\|_{\Lambda^{k,\sigma}_{\delta}(\Omega_R^j)} + \|u\|_{C^{k,\sigma}(\hat\Omega_R/2)}
    \end{equation*}
    with some vector $\vec\delta \in [\sigma,\infty)^d$. Analogously the norm $\|\cdot\|_{C^{k,\sigma}_{\vec\delta}(\Omega)}$ is defined.
    The corresponding function spaces are defined by
    \[
      \Lambda^{k,\sigma}_{\vec\delta}(\Omega):=\overline{C^\infty_0(\overline\Omega\setminus \mathcal S)}
      \vphantom{X}^{\|\cdot\|_{\Lambda^{k,\sigma}_{\vec\delta}(\Omega)}},\qquad
      C^{k,\sigma}_{\vec\delta}(\Omega):=\overline{C^\infty_0(\overline\Omega)}
      \vphantom{X}^{\|\cdot\|_{C^{k,\sigma}_{\vec\delta}(\Omega)}},
    \]
    where $\mathcal S:=\{\boldsymbol c_j\colon j=1,\ldots,d\}$.
    The corresponding trace spaces are endowed with the norm
    \begin{equation}\label{eq:hoelder_trace}
      \|u\|_{\Lambda^{k,\sigma}_{\vec\delta}(\Gamma)} := \inf\{\|\tilde u\|_{\Lambda^{k,\sigma}_{\vec\delta}(\Omega)}\colon \tilde u|_{\Gamma\setminus\{\boldsymbol c_j, j\in \mathcal C\}} \equiv u\},
    \end{equation}
    and analogously for $C^{k,\sigma}_{\vec\delta}(\Gamma)$.
    
Next, we establish a regularity result for weighted Sobolev spaces.
\begin{lemma}\label{lem:regularity_W22}
  Let $f\in W^{0,2}_{\vec\beta}(\Omega)$ and $g\in W^{3/2,2}_{\vec\beta}(\Gamma)$
  with $\vec\beta\in [0,1)^d$ satisfying $\beta_j > 1-\lambda_j$ for all $j\in\mathcal C$.
  Then, the solution of \eqref{eq:bvp} belongs to $W^{2,2}_{\vec\beta}(\Omega)$. In case of $g\in V^{3/2,2}_{\vec\beta}(\Gamma)$ the function $y$ belongs to $V^{2,2}_{\vec\beta}(\Omega)$.
\end{lemma}
\begin{proof}
  The regularity result for $V$-spaces can be deduced from \cite[Theorem 1.4.3]{KMR01}.
  Note that this result holds even for a larger range
  of the weights, this is, $\beta_j\in (1-\lambda_j,1+\lambda_j)$.
  From this result we infer the solvability in $W$-spaces as each function $y\in W^{2,2}_{\vec\beta_j}(\Omega_R^j)$ with $\beta_j\in(0,1)$ can be decomposed into $y_0 + p$ with a constant $p=g(x_j)$
  and $y_0\in V^{2,2}_{\beta_j}(\Omega_R^j)$. This is basically the idea which leads to
  \cite[Theorem 4.\S5.11]{NP94} from which we could conclude the same result.
\end{proof}
An analogue of this result is true for the weighted H\"older spaces introduced
above. This is used to show boundedness of the
solution of \eqref{eq:bvp} in a weighted $W^{2,\infty}$-space.
\begin{lemma}\label{lem:reg_hoelder}
 Assume that $f\in \Lambda^{0,\sigma}_{\vec\delta}(\Omega)$ and $g\in
 \Lambda^{2,\sigma}_{\vec\delta}(\Gamma)$ with $\sigma\in (0,1]$
 and weights $\vec\delta\in (\sigma,2+\sigma)^d$ satisfying 
 $2-\lambda_j>\delta_j-\sigma$ for $j\in\mathcal C$.
 Moreover, we exclude the case $\delta_j-\sigma=1$. 
 Then, the solution $y$ of \eqref{eq:bvp} belongs to
 $\Lambda^{2,\sigma}_{\vec\delta}(\Omega)$ and depends continuously on the input
 data. This result remains true when replacing $\Lambda$ by $C$.
\end{lemma}
\begin{proof}  
  The proof for the regularity in $\Lambda$-spaces can be deduced from
  \cite[Theorem 1.4.5]{KMR01}. 
  In order to show the regularity result in the weighted $C$-spaces we
  basically follow the ideas used in \cite[Lemma 3.13]{Pfe}.
  First, introduce the numbers $\nu_j$, $j\in\mathcal C$, such that $\nu_j <
  \delta_j-\sigma < \nu_j+1$. 
 Then, we split the solution into
\begin{equation*}
  y = u + \sum_{j=1}^d \eta_j\, p_j,
\end{equation*}
with smooth cut-off functions $\eta_j$ satisfying $\eta_j\equiv 1$ in
$\Omega_{R/2}^j$ and $\supp{\eta_j}\subset \overline{\Omega_R^j}$ for
all $j\in\mathcal C$, and polynomials $p_j$ of order not greater than $1-\nu_j$.
The key idea is to show that $u$ belongs to a weighted $\Lambda$-space
and the desired result follows from certain relations between $C$- and
$\Lambda$-spaces.
By a reformulation of the boundary value problem, we confirm that $u$ solves 
\begin{align*}
 -\Delta u &= f + \sum_{j=1}^d \left(\Delta\eta_j\, p_j + 2\,\nabla\eta_j\cdot\nabla p_j\right) := F &&\mbox{in}\ \Omega,\\
u &= g - \sum_{j=1}^d \eta_j\, p_j:=G  && \mbox{on}\ \Gamma.
\end{align*}
Our aim is to show that $u$ belongs to $\Lambda^{2,\sigma}_{\vec\delta}(\Omega)$ which would
follow under the assumption $F\in\Lambda^{0,\sigma}_{\delta}(\Omega)$ and
$G\in\Lambda^{2,\sigma}_{\vec\delta}(\Gamma)$. To achieve this, 
we have to construct the polynomials $p_j$ appropriately and therefore, we 
we define the projection 
\begin{equation}\label{eq:def_qk}
  q_k(v;\boldsymbol c_j)(x) := \sum_{|\alpha|=0}^k \frac{1}{|\alpha|!}(D^\alpha
  v)(\boldsymbol c_j)\,(x-\boldsymbol c_j)^\alpha
\end{equation}
for $j\in\mathcal C$ and $k\in \mathbb N_0$.
In a similar way we construct a projection for functions defined on the
boundary by means of $q_k^\partial(v,\boldsymbol c_j):=\gamma_0
q_k(\tilde v;\boldsymbol c_j)$, where $\tilde v$ is an arbitrary extension of
$v$ and $\gamma_0$ is the trace operator. The polynomial
$q_k^\partial(v;\boldsymbol c_j)$ is independent of the extension $\tilde v$, and hence, there holds
$\gamma_0 q_k(y; \boldsymbol c_j) = q_k^\partial(g;\boldsymbol c_j)$ as $y\equiv g$ on $\Gamma$.
In the following we use the choice $p_j := q_{1-\nu_j}(y;\boldsymbol c_j)$.
That $F$ belongs to $\Lambda^{0,\sigma}_{\vec\delta}(\Omega)$ is obvious, as
$f$ is assumed to be contained in $C^{0,\sigma}_{\vec\delta}(\Omega)$ and this
space is equivalent to $\Lambda^{0,\sigma}_{\vec\delta}(\Omega)$ if
$\vec\delta>0$, see the arguments before Lemma 6.7.1 in
\cite{MR10}.
Moreover, the cut-off functions $\eta_j$ are constant in the neighborhood of the corners
and thus, the products $\nabla\eta_j\cdot\nabla p_j$ and $\Delta \eta_j p_j$
belong trivially to that space.
Consequently, we get
\begin{equation}\label{eq:hoelder_F}
 \|F\|_{\Lambda^{0,\sigma}_{\vec\delta}(\Omega)} \le c\,\left(
 \|f\|_{C^{0,\sigma}_{\vec\delta}(\Omega)} + \sum_{j=1}^d
 \|p_j\|_{C^1(\Omega_R^j)} \right).
\end{equation}
With the definition \eqref{eq:def_qk} and the
imposed Dirichlet boundary conditions, taking into account
$p_j|_\Gamma = q_{1-\nu_j}^\partial(g; \boldsymbol c_j)$, $j\in\mathcal C$,
we deduce
\begin{align}\label{eq:bound_pj}
  \|p_j\|_{C^1(\Omega_R^j)}
  &\le c\, \sum_{|\alpha|=0}^{1-\nu_j} |D^\alpha g(\boldsymbol c_j)|
    \le c\, \|\tilde g\|_{C^{1-\nu_j,\varepsilon}_{0}(\Omega_R^j)} \nonumber \\
  & \le c\, \|\tilde g\|_{C^{2,\sigma}_{1+\nu_j+\sigma-\varepsilon}(\Omega_R^j)}
    \le c\, \|\tilde g\|_{C^{2,\sigma}_{\delta_j}(\Omega_R^j)}
    = c\, \|g\|_{C^{2,\sigma}_{\delta_j}(\Gamma_R^j)}
\end{align}
for $\delta_j - \sigma < 1+\nu_j$ and sufficiently small $\varepsilon>0$,
and $\tilde g$ is a suitable extension of $g$, see \eqref{eq:hoelder_trace}.
The second and third step follow from the equivalence of $C^{k,\sigma}_{0}(\Gamma_R^j)$ and $C^{k,\sigma}(\Gamma_R^j)$ stated in \cite[Lemma 6.7.2]{MR10} and an embedding theorem for weighted H\"older spaces,
see \cite[Lemma 6.7.1]{MR10}. The embedding used in the last step is trivial.

The property $G\in \Lambda^{2,\sigma}_{\vec\delta}(\Gamma)$ follows from 
\cite[Theorem 6.7.6]{MR10} which provides the a priori estimate 
\begin{equation}\label{eq:hoelder_G}
  \|G\|_{\Lambda^{2,\sigma}_{\vec\delta}(\Gamma)}
  = \|g-\sum_{j=1}^d \eta_j\, p_j\|_{\Lambda^{2,\sigma}_{\vec\delta}(\Gamma)} \le 
  c\,\|g\|_{C^{2,\sigma}_{\vec\delta}(\Gamma)}.
\end{equation}
The regularity result proved in \cite[Theorem 1.4.5(2)]{KMR01} then guarantees $u\in
\Lambda^{2,\sigma}_{\vec\delta}(\Omega)$, and with
the triangle inequality, the trivial estimate
$\|v\|_{C^{2,\sigma}_{\vec\delta}(\Omega)} \le c\,
\|v\|_{\Lambda^{2,\sigma}_{\vec\delta}(\Omega)}$ for $v\in
\Lambda^{2,\sigma}_{\vec\delta}(\Omega)$ and \eqref{eq:bound_pj}
we infer
\begin{align*}
  \|y\|_{C^{2,\sigma}_{\vec\delta}(\Omega)} &\le \|u\|_{C^{2,\sigma}_{\vec\delta}(\Omega)} + \sum_{j=1}^d
  \|\eta_j\, p_j\|_{C^{1}(\Omega)} \\
  &\le c\,\left(\|u\|_{\Lambda^{2,\sigma}_{\vec\delta}(\Omega)} + \|g\|_{C^{2,\sigma}_{\vec\delta}(\Gamma)}\right).
\end{align*}
An a priori estimate for the weighted $\Lambda$-norm of $u$ can be concluded from
\cite[Theorem 1.4.5(1)]{KMR01}, which leads to
\begin{align}\label{eq:a_priori_Lambda}
  \|y\|_{C^{2,\sigma}_{\vec\delta}(\Omega)}
  &\le c\,\left(\|F\|_{\Lambda^{0,\sigma}_{\vec\delta}(\Omega)}
    + \|G\|_{\Lambda^{2,\sigma}_{\vec\delta}(\Gamma)} + \|u\|_{L^1(\Omega)} + \|g\|_{C^{2,\sigma}_{\vec\delta}(\Gamma)}\right)\nonumber\\
  &\le c\,\left(\|f\|_{C^{0,\sigma}_{\vec\delta}(\Omega)} + \|g\|_{C^{2,\sigma}_{\vec\delta}(\Gamma)}
    + \|u\|_{L^1(\Omega)}\right),
\end{align}
where the second step follows from the estimates \eqref{eq:hoelder_F}
and \eqref{eq:hoelder_G}.
The $L^1(\Omega)$ norm of $u$ can be bounded by the
$V^{2,2}_{\vec\beta}(\Omega)$-norm with weights $\beta_j =
\max\{0,\delta_j-\sigma-1\}+\varepsilon$, $j\in\mathcal C$, and $\varepsilon>0$ sufficiently small.
Using the norm equivalence from \cite[Theorem 5.6]{NP94} or \cite[Lemma 6.2.12]{MR10} we arrive at
\begin{align}\label{eq:L1_u}
  \|u\|_{L^1(\Omega)} 
  \le c\, \|u\|_{V^{2,2}_{\vec\beta}(\Omega)}
  \le c\,\|y\|_{W^{2,2}_{\vec\beta}(\Omega)}.
  \end{align}
With Lemma \ref{lem:regularity_W22} and embeddings of $W$- into $C$-spaces,
see e.\,g.\ \cite[Lemma 2.39]{Pfe}, we deduce
\begin{equation}\label{eq:y_W22}
\|y\|_{W^{2,2}_{\vec\beta}(\Omega)} \le
c\left(\|f\|_{C^{0,\sigma}_{\vec\delta}(\Omega)} + \|g\|_{C^{2,\sigma}_{\vec\delta}(\Gamma)}\right).
\end{equation}
The desired a priori estimate follows after insertion of 
\eqref{eq:L1_u} and \eqref{eq:y_W22} into \eqref{eq:a_priori_Lambda}.
\end{proof}
The regularity results in weighted H\"older spaces allow us to extend the assertion
of Lemma \ref{lem:regularity_W22} to 
$L^\infty$ based norms. This is a simple conclusion from the definition of the spaces $V$ and $\Lambda$
as well as $W$ and $C$.
\begin{corollary}\label{cor:reg_W2inf}
  Assume that $\vec\delta\in (\sigma,2+\sigma)^d$ satisfies the assumptions
  of Lemma \ref{lem:reg_hoelder}. Let $\vec\gamma\in (0,2)^d$ be a weight vector defined by
  $\gamma_j:=\delta_j-\sigma$ for $j\in\mathcal C$.
  \begin{enumerate}[i)]
  \item If $f\in \Lambda^{0,\sigma}_{\vec\delta}(\Omega)$ and
    $g\in \Lambda^{2,\sigma}_{\vec\delta}(\Gamma)$,
    the solution $y$ of \eqref{eq:bvp} belongs to $V^{2,\infty}_{\vec\gamma}(\Omega)$.
  \item If $f\in C^{0,\sigma}_{\vec\delta}(\Omega)$ and
    $g\in C^{2,\sigma}_{\vec\delta}(\Gamma)$,
    the solution $y$ of \eqref{eq:bvp} belongs to $W^{2,\infty}_{\vec\gamma}(\Omega)$.
  \end{enumerate}
\end{corollary}

\section{Error estimates for normal derivatives}\label{sec:bvp}

In this section we consider a finite element discretization for the weak form of the boundary value problem \eqref{eq:bvp} which reads
\begin{equation*}
 y|_\Gamma\equiv g,\quad (\nabla y,\nabla v)_{L^2(\Omega)^2} = (f,v)_{L^2(\Omega)} \qquad \forall v\in H^1_0(\Omega).
\end{equation*}
Therefore, let $\{\mathcal T_h\}_{h>0}$ be a quasi-uniform family of shape-regular triangulations of $\Omega$,
which are feasible in the sense of \cite[Section 5]{Cia91}. The parameter $h$ denotes the maximal diameter of all elements from $\mathcal T_h$.
The trial and test spaces are defined by
\begin{equation*}
 V_h:=\{v_h\in C(\overline\Omega)\colon v_h|_T\in \mathcal P_1\ \mbox{for all}\ T\in \mathcal T_h\},\quad V_{0h}:=V_h\cap H^1_0(\Omega).
\end{equation*}
Moreover, the traces of function from $V_h$ belong to the space 
\begin{equation*}
V_h^\partial:=\{w_h\in C(\Gamma)\colon w_h= v_h|_\Gamma\ \mbox{for some}\ v_h\in V_h\}. 
\end{equation*}
The finite-element approximation $y_h\in V_h$ of $y$ is defined by
\begin{equation}\label{eq:fe_form}
 y_h|_\Gamma \equiv g_h,\quad (\nabla y_h,\nabla v_h)_{L^2(\Omega)^2}= (f,v_h)_{L^2(\Omega)}\qquad \forall v_h\in V_{0h},
\end{equation}
where $g_h\in V_h^\partial$ is some appropriate interpolation or projection of $g$.
In the following, $g_h$ will be the $L^2(\Gamma)$-projection of $g$ onto $V_h^\partial$, this is, $g_h:=Q_h(g)$.
Moreover, we denote by 
\begin{equation}\label{eq:Ih}
 I_h\colon C(\overline\Omega)\to V_h,\qquad [I_h u](x) = \sum_{i=1}^{N} u(x_i)\,\varphi_i(x)
\end{equation}
the nodal interpolant.
Here, $x_i\in\Omega$, $i=1,\ldots,N$, denote the nodes of $\mathcal T_h$ and $\{\varphi_i\}_{i=1}^N$ the nodal basis of $V_h$. Moreover, we will use a slightly modified interpolant defined by
\begin{equation}\label{eq:Ih_tilde}
  \tilde I_h y = I_h y + \mathcal E_h(Q_h g-I_h g),
\end{equation}
where $\mathcal E_h\colon V_h^\partial\to V_h$ is the zero extension which vanishes in the interior nodes of $\mathcal T_h$. For functions $y\in C(\overline\Omega)$ with $y|_\Gamma = g$, the interpolant fulfills the essential property $[\tilde I_h y]|_\Gamma \equiv Q_h g$ that is needed for instance in the proof
of a C\'ea-Lemma. As the local interpolation error estimates will frequently depend on the distance
to the corners, we introduce the notation
\begin{equation*}
  r_{j,T} = \inf_{x\in T}|x-\boldsymbol c_j|\qquad j\in\mathcal C,\ T\in\mathcal T_h.
\end{equation*}

We start our investigations with an interpolation error estimate for the
boundary datum $g$. 
\begin{lemma}\label{lem:int_error_g}
  Let be given weight vectors $\vec\alpha\in[0,1/2)^d$ and $\vec\gamma\in[0,3/2)^d$.
Then, the interpolation error estimates 
\begin{align*}
  \|g-I_h g\|_{L^2(\Gamma)}+ h^{1/2}\,\|g-I_h g\|_{H^{1/2}(\Gamma)} &\le c\,
  h^{2-\overline{\vec\gamma}}\,|g|_{W^{2,2}_{\vec\gamma}(\Gamma)},\\
  \|g-I_h g\|_{L^2(\Gamma)} + h^{1/2}\,\|g-I_h g\|_{H^{1/2}(\Gamma)}&\le c\,
  h^{3/2-\overline{\vec\alpha}}\,\|g\|_{W^{3/2,2}_{\vec\alpha}(\Gamma)},
\end{align*}
are valid, provided that $g$ possesses the regularity demanded by the
right-hand side.
\end{lemma}
\begin{proof}
  The first estimate can be deduced from \cite[Lemma 3.2.4]{Win15}. There, the desired error
  estimate in the $L^2(\Gamma)$- and $H^1(\Gamma)$-norm is proved. The estimate in $H^{1/2}(\Gamma)$
  follows from an interpolation argument.
  
  To show the second estimate we will reuse existing interpolation error
  estimates exploiting regularity in weighted $V$-spaces.
  To this end, we split up the function $g$ by means of $g=g_0+\eta_j\, p_j$
  with $g_0\in V^{3/2,2}_{\vec\alpha}(\Gamma)$, certain constants $p_j\in \mathbb R$,
  $j\in\mathcal C$, and smooth cut-off functions $\eta_j=\eta_j(|x-\boldsymbol c_j|)$
  satisfying
  \[\eta_j|_{\Omega_{R/2}^j}\equiv 1,\quad \supp(\eta_j)\subset \overline{\Omega_R^j}
    \quad\mbox{and}\quad \|D^\alpha \eta_j\|_{L^\infty(\Omega)}\le c \ \forall |\alpha|\le 2.\]
  Note that the nodal interpolant preserves the functions $\eta_j\, p_j$ near the corners.
  Hence, it suffices to prove an estimate for $g_0$.
  In order to derive local interpolation error estimates we denote by $\hat E:=(0,1)$ the reference
  interval, and by $F_E\colon \hat E\to E$ the affine reference transformation.
  Moreover, we write $\hat v(\hat x) = v(F_E(\hat x))$ for all $\hat x\in \hat E$.
  The norms of the weighted Sobolev spaces on the reference element, $V^{k-1/2,2}_{\alpha}(\hat E)$,
  are defined analogous to the global norms introduced in Section \ref{sec:weighted_spaces}
  but the weight function is defined by $\hat r(\hat x):=|\hat x|$.
  For elements $E\in\mathcal E_h$ touching the corner $\boldsymbol c_j$, $j\in\mathcal C$,
  there holds the property $r_j(F_E(\hat x)) \sim h\,\hat r(\hat x)$.
  
  For all elements $E\in\mathcal E_h$ with $r_{j,E}=0$ for some $j\in\mathcal C$, we obtain
  the estimate
  \begin{align*}
    \|g_0 - I_h g_0\|_{L^2(E)} &\le c\, |E|^{1/2}\,\|\hat g_0\|_{L^\infty(\hat E)}
    \le c\, |E|^{1/2}\,\|\hat g_0\|_{V^{3/2,2}_{\alpha_j}(\hat E)}\\
    &\le c\, h^{3/2-\alpha_j}\,\|g_0\|_{V^{3/2,2}_{\alpha_j}(E)},
    \end{align*}
    which follows from the arguments used in the proof of \cite[Lemma 4.5]{APR12}.
    In case of $E\subset \Omega_{R/2}^j$ for some $j\in\mathcal C$ and $r_{j,E} > 0$  we deduce
    \begin{equation*}
      \|g_0-I_h g_0\|_{L^2(E)} \le c\, h^{3/2}\,|g_0|_{H^{3/2}(E)}
      \le c\, h^{3/2-\alpha_j}\,\|g_0\|_{V^{3/2}_{\alpha_j}(E)},
  \end{equation*}
  where the argument used in the last step can also be found in \cite[Lemma 4.5]{APR12}.
  Far away from the corners, i.\,e.\
  $r_{j,E} > 1/4$ for all $j\in\mathcal C$, 
  we can use a standard estimate to get $\|g_0-I_h g\|_{L^2(E)} \le c\, h^{3/2}\,|g_0|_{H^{3/2}(E)}$.
  Combining the previous estimates and using a standard estimate for the error terms $p_j\,\eta_j - I_h(p_j\,\eta_j)$ yields
  \begin{align}\label{eq:interpolation_V_to_W}
    \|g - I_h g\|_{L^2(\Gamma)}
    &\le c\, \left(\|g_0-I_h g_0\|_{L^2(\Gamma)} + h^2\,\sum_{j\in\mathcal C} |p_j|\right) \nonumber\\
    &\le
      ch^{3/2-\overline{\vec\alpha}}\left(\|g_0\|_{V^{3/2,2}_{\vec\alpha}(\Gamma)}
      + \sum_{j\in\mathcal C} |p_j|\right)
      \le c\, h^{3/2-\overline{\vec\alpha}}\,\|g\|_{W^{3/2,2}_{\vec\alpha}(\Gamma)}.     
  \end{align}
  The last step is a consequence of the norm equivalence stated in \cite[Ch.\ 4, Theorem 5.7]{NP94}.

  The estimate in the $H^{1/2}(\Gamma)$-norm follows from an interpolation argument between
  estimates in $L^2(\Gamma)$ and $H^1(\Gamma)$.
  To show an estimate in $H^1(\Gamma)$ we derive local estimates first.
  For all elements $E\in\mathcal E_h$ with $r_{j,E} = 0$ we obtain
  \begin{align*}
    \|g_0-I_h g_0\|_{H^1(E)}&\le c\, h^{-1}\,|E|^{1/2}\,\|\hat g_0\|_{H^1(\hat E)}
    \le c\, h^{-1}\,|E|^{1/2}\,\|\hat g_0\|_{V^{3/2,2}_{\alpha_j}(\hat E)}\\
    &\le c\, h^{1/2-\alpha_j}\,\|g_0\|_{V^{3/2,2}_{\alpha_j}(E)},
  \end{align*}
  where the second step is an application of the embedding $V^{3/2,2}_{\alpha_j}(\hat E)\hookrightarrow H^1(\hat E)$, which is valid for $\alpha_j < 1/2$.
  Otherwise, if $E\subset\Omega_{R/2}^j$ and $r_{j,E} > 0$, we obtain with similar arguments
  \begin{equation*}
    \|g_0-I_h g_0\|_{H^1(E)} \le c\, h^{1/2}\,|g_0|_{H^{3/2}(E)} \le c\, h^{1/2-\alpha_j}\,\|g_0\|_{V^{3/2,2}_{\alpha_j}(E)}.
  \end{equation*}
  In the far interior, i.\,e.\ for $E\in\mathcal E_h$ with $r_{j,E}> 1/4$ for all $j\in\mathcal C$,
  we can use a standard estimate exploiting $H^{3/2}$-regularity. Summation over all elements
  $E\in\mathcal E_h$ and an interpolation argument lead to the desired estimate for $g_0-I_h g_0$
  in the $H^{1/2}(\Gamma)$-norm. With the splitting $g = g_0+\sum_{j\in\mathcal C} p_j\, \eta_j$ we get
  an estimate for $g-I_h g$ when using the arguments from \eqref{eq:interpolation_V_to_W}.
\end{proof}

Using the ideas of \cite{BCD04} we can derive error estimates for the approximate solutions
$y_h$ in the norms $H^1(\Omega)$ and $L^2(\Omega)$ . However,
in this reference $H^2(\Gamma_j)$-regularity ($j\in\mathcal C$) for the
Dirichlet datum $g$ is assumed.
As we deal with optimal Dirichlet control problems, 
the boundary datum for the state is the control function which might be less regular.
Thus, we repeat the proof assuming less regularity for $g$ in some weighted Sobolev space.
\begin{lemma}\label{lem:error_estimate_l2}
  Assume that $y\in W^{2,2}_{\vec\alpha}(\Omega)$ and $g\in W^{3/2,2}_{\vec\alpha}(\Gamma)$
  with a weight vector $\vec\alpha\in [0,1/2)^d$.
  Moreover, let $g_h:= Q_h g$.  
  Then, the solution $y_h\in V_h$ of \eqref{eq:fe_form} satisfies the error estimates
 \begin{equation}\label{eq:fe_error_l2}
   \begin{aligned}
        \|y-y_h\|_{H^1(\Omega)}
   &\le c\, h^{1-\bar{\vec\alpha}}\left(|y|_{W^{2,2}_{\vec\alpha}(\Omega)}
     + \|g\|_{W^{3/2,2}_{\vec\alpha}(\Gamma)}\right),\\
   \|y-y_h\|_{L^2(\Omega)}
   &\le c\, h^{3/2-\overline{\vec\alpha} + \varepsilon(\Omega)}\left(|y|_{W^{2,2}_{\vec\alpha}(\Omega)}
       + \|g\|_{W^{3/2,2}_{\vec\alpha}(\Gamma)}\right),
   \end{aligned}
 \end{equation}
 with some sufficiently small $\varepsilon(\Omega)\in (0,1/2]$ depending on the
 opening angles of the corners of $\Omega$. For convex domains, the choice $\varepsilon(\Omega)=1/2$ is possible.
\end{lemma}
\begin{proof}
  First, we derive the error estimate in $H^1(\Omega)$.
  We apply the error equation
  $(\nabla(y-y_h), \nabla (\tilde I_h y - y_h))_{L^2(\Omega)^2} = 0$
  and obtain
  \begin{align*}
    \|\nabla(y-y_h)\|_{L^2(\Omega)}^2 &= (\nabla(y-y_h), \nabla(y-I_h y))_{L^2(\Omega)^2}\\
    &+ (\nabla (y-y_h),\nabla\mathcal E_h (I_h g - Q_h g))_{L^2(\Omega)^2}.                                        
  \end{align*}
  With a standard interpolation error estimate exploiting weighted regularity,
  see e.\,g\ \cite[Lemma 3.31]{Pfe}, we get 
  \begin{equation*}
    (\nabla(y-y_h), \nabla(y-I_h y))_{L^2(\Omega)^2} \le c\, h^{1-\overline{\vec\alpha}}\,|y|_{W^{2,2}_{\vec\alpha}(\Omega)}\,\|\nabla(y-y_h)\|_{L^2(\Omega)}.
  \end{equation*}
  Note that the zero extension satisfies
  $\|\nabla\mathcal E_h \phi_h\|_{L^2(\Omega)} \le c\, h^{-1/2}\,\|\phi_h\|_{L^2(\Gamma)}$,
  see e.\,g.\ \cite[Lemma 3.3]{MRV13}. Thus, together with Lemma \ref{lem:int_error_g}
  we obtain an estimate for the second term 
  \begin{align*}
    (\nabla (y-y_h),\nabla\mathcal E_h (I_h g - Q_h g))
    &\le c\, h^{-1/2}\,\|g-I_h g\|_{L^2(\Gamma)}\,\|\nabla(y-y_h)\|_{L^2(\Omega)}\\
    &\le c\, h^{1-\overline{\vec\alpha}}\,\|g\|_{W^{3/2,2}_{\vec\alpha}(\Gamma)}\,\|\nabla(y-y_h)\|_{L^2(\Omega)}.    
  \end{align*}
  
  In order to derive an estimate in the $L^2(\Omega)$-norm we use a duality argument.
  Let $w\in H^1_0(\Omega)$ be the weak solution of $-\Delta w = y-y_h$ in $\Omega$. 
  With partial integration, the orthogonality of the $L^2(\Gamma)$-projection $Q_h$,
  the estimate in the $H^1(\Omega)$-norm 
  and Lemma \ref{lem:int_error_g}, we obtain for sufficiently small $\varepsilon(\Omega)\in (0,1/2]$
 \begin{align*}
   &\|y-y_h\|_{L^2(\Omega)}^2
   = (y-y_h, -\Delta w)_{L^2(\Omega)}\\
   &\quad = (\nabla(y-y_h), \nabla (w-I_h w))_{L^2(\Omega)^2}
     - (g-g_h,\partial_nw - Q_h(\partial_n w))_{L^2(\Gamma)} \\
   &\quad\le c\, h^{3/2-\overline{\vec\alpha}+\varepsilon(\Omega)}\left(|y|_{W^{2,2}_{\vec\alpha}(\Omega)}\,\|w\|_{H^{3/2+\varepsilon(\Omega)}(\Omega)}
     + \|g\|_{W^{3/2,2}_{\vec\alpha}(\Gamma)}\,\|\partial_n w \|_{H^{\varepsilon(\Omega)}(\Omega)}\right).
 \end{align*}
 The assertion follows from the a priori estimate
 \[\|\partial_n w\|_{H^{\varepsilon(\Omega)}(\Gamma)} \le c\, \|w\|_{H^{3/2+\varepsilon(\Omega)}(\Omega)} \le
 c\,\|y-y_h\|_{L^2(\Omega)}.\]
\end{proof}

The aim in the remainder of this section is to derive error estimates for the variational normal derivative of the approximate solution $y_h$.
Motivated by Green's identity this is defined by
\begin{equation}\label{eq:discrete_normal_deriv}
 \partial_n^h y_h \in V_h^\partial\colon\quad (\partial_n^h y_h,v_h)_{L^2(\Gamma)} = (\nabla y_h,\nabla v_h)_{L^2(\Omega)^2} - (f,v_h)_{L^2(\Omega)}\quad\forall v_h\in V_h.
\end{equation}
Note that both the left- and right-hand side are zero for test functions from $V_{0h}$.
Hence, in order to compute $\partial_n^h y_h$, it suffices to test the equation \eqref{eq:discrete_normal_deriv} with the nodal basis functions that belong to the boundary nodes.

We start our considerations with an existence and stability result.
\begin{lemma}\label{lem:stability_normal}
 For arbitrary input data $f\in L^2(\Omega)$, $g_h\in V_h^\partial$, the variational normal derivative $\partial_n^h y_h\in V_h^\partial$ defined by \eqref{eq:discrete_normal_deriv}
 exists, is unique, and satisfies the estimate
 \begin{equation*}
  \|\partial_n^h y_h\|_{H^{-1/2}(\Gamma)} \le c\,\left(\|f\|_{L^2(\Omega)} + \|g_h\|_{H^{1/2}(\Gamma)}\right).
 \end{equation*}
\end{lemma}
\begin{proof}
  In the following $B_h\colon V_h^\partial \to V_h$ is the discrete harmonic extension
  which satisfied the well-known estimate $\|B_h v_h\|_{H^1(\Omega)} \le c\, \|v_h\|_{H^{1/2}(\Gamma)}$.
  Together with the discrete stability of functions from $V_h^\partial$ in $H^{1/2}(\Gamma)$ 
  as well as \eqref{eq:fe_form} we obtain
 \begin{align}\label{eq:discrete_stability}
   \|\partial_n^hy_h\|_{H^{-1/2}(\Gamma)}
   &\le c\, \sup_{\genfrac{}{}{0pt}{}{v_h\in V_h^\partial}{v_h\not\equiv 0}} \frac{(\partial_n^h y_h,v_h)_{L^2(\Gamma)}}{\|v_h\|_{H^{1/2}(\Gamma)}} \nonumber\\
   &\le c\, \sup_{\genfrac{}{}{0pt}{}{v_h\in V_h^\partial}{v_h\not\equiv 0}} \frac{(\nabla y_h,\nabla B_hv_h)_{L^2(\Omega)^2} - (f,B_hv_h)_{L^2(\Omega)}}{\|B_hv_h\|_{H^1(\Omega)}} \nonumber \\
   &\le c\, \left(\|\nabla y_h\|_{L^2(\Omega)} + \|f\|_{L^2(\Omega)}\right).
 \end{align}
 The a priori estimate $\|y_h\|_{H^1(\Omega)} \le c\,(\|f\|_{L^2(\Omega)} + \|g_h\|_{H^{1/2}(\Gamma)})$ implies the assertion.
\end{proof}
Next, we show an error estimate for the variational normal derivative,
for which we exploit the $W^{2,2}_{\vec\alpha}(\Omega)$-regularity of the solution.
The result of the following theorem is sharp for non-convex domains $\Omega$,
and also for convex domains when the solution is not more regular than $H^2(\Omega)$
(this happens e.\,g.\ when the right-hand side belongs to $L^2(\Omega)$, but not to $L^p(\Omega)$ with $p>2$). Later, we prove an estimate which promises a higher convergence rate for convex domains, provided that the solution belongs to $W^{2,\infty}_{\vec\beta}(\Omega)$.
\begin{theorem}\label{thm:main_result_nonconvex}
  Let $\Omega$ be an arbitrary polygonal domain.
  Moreover, let $g_h = Q_h g$. Under the assumptions $y\in
   W^{2,2}_{\vec\alpha}(\Omega)$ 
   and $g\in W^{3/2}_{\vec\alpha}(\Gamma)$ with $\vec\alpha\in [0,1/2)^d$, there
   holds the error estimate
   \begin{equation*}
     \|\partial_n y - \partial_n^h y_h\|_{H^{-1/2}(\Gamma)} \le c\, h^{1-\overline{\vec\alpha}}
     \left(\|y\|_{W^{2,2}_{\vec\alpha}(\Omega)}
       + \|g\|_{W^{3/2,2}_{\vec\alpha}(\Gamma)}\right).
   \end{equation*}
 \end{theorem}
 \begin{proof}
     Using the triangle inequality we split up the norm into an error term for the $L^2(\Gamma)$-projection onto $V_h^\partial$, and a fully discrete term, this is
  \begin{equation*}
    \|\partial_n y- \partial_n^h y_h\|_{H^{-1/2}(\Gamma)} \le \|\partial_n y - Q_h(\partial_n y)\|_{H^{-1/2}(\Gamma)} + \|Q_h(\partial_n y) - \partial_n^h y_h\|_{H^{-1/2}(\Gamma)}.
  \end{equation*}
  With a standard duality argument we obtain for the first term
  \begin{equation}\label{eq:int_error_duality}
    \|\partial_n y - Q_h(\partial_n y)\|_{H^{-1/2}(\Gamma)} \le c\,
    h^{1/2}\,\|\partial_n y - Q_h (\partial_n y)\|_{L^2(\Gamma)}.
  \end{equation}
  Next, we show a best-approximation error estimate in the $L^2(\Gamma)$-norm.
  To this end, we use the splitting splitting $y=y_0+p_j\,\eta_j$,
  see e.\,g.\ \cite[Theorem 5.6(2)]{NP94}, with a
  function $y_0\in V^{2,2}_{\vec\alpha}(\Omega)$, certain constants $p_j$,
  $j\in\mathcal C$, and smooth cut-off functions $\eta_j=\eta_j(|x-\boldsymbol c_j|)$ which are equal to one near $\boldsymbol c_j$ and have support contained in $\overline\Omega_R^j$. A similar argument has been already used in the proof
  of Lemma \ref{lem:reg_hoelder}. For functions belonging to $V^{2,2}_{\vec\alpha}(\Omega)$
  the estimate 
  \begin{equation*}
    \|\partial_n y_0 - C_h(\partial_n y_0)\|_{L^2(\Gamma)}
    \le c\, h^{1/2-\overline{\vec\alpha}}\,\|y_0\|_{V^{2,2}_{\vec\alpha}(\Omega)}
  \end{equation*}
  can be found in the proof of Theorem 9 in \cite{PW17} for some
  Cl\'ement-type interpolation operator $C_h\colon L^1(\Gamma)\to V_h^\partial$.
  Note that $\partial_n (p_j\,\eta_j)$ and its interpolant vanish and thus, we easily deduce an estimate for the function $y\in W^{2,2}_{\vec\alpha}(\Omega)$.
  Moreover, due to norm equivalences of $V$- and $W$-spaces \cite[Theorem 5.6(2)]{NP94},
  we obtain
  $\|y_0\|_{V^{2,2}_{\vec\alpha}(\Omega)} + \sum_{j\in\mathcal C} |y(\boldsymbol c_j)| \sim \|y\|_{W^{2,2}_{\vec\alpha}(\Omega)}$,
  which leads together with the previous estimate and \eqref{eq:int_error_duality} to
  \begin{equation*}
    \|\partial_n y - Q_h(\partial_n y)\|_{H^{-1/2}(\Gamma)} \le c\, h^{1-\overline{\vec\alpha}}\,\|y\|_{W^{2,2}_{\vec\alpha}(\Omega)}.
  \end{equation*}

 With the discrete stability used already in \eqref{eq:discrete_stability},
 the definition of $\partial_n^h$ from \eqref{eq:discrete_normal_deriv},
 orthogonality of the $L^2(\Gamma)$-projection $Q_h$
 and Greens identity, we deduce
 \begin{align}\label{eq:discrete_part}
   \|Q_h(\partial_n y) - \partial_n^h y_h\|_{H^{-1/2}(\Gamma)} 
   &\le c\,\sup_{\genfrac{}{}{0pt}{}{\varphi_h\in V_h^\partial}{\varphi_h\not\equiv 0}} \frac{(Q_h(\partial_n y) - \partial_n^h y_h,\varphi_h)_{L^2(\Gamma)}}{\|\varphi_h\|_{H^{1/2}(\Gamma)}}\nonumber\\
   &\le c\,\sup_{\genfrac{}{}{0pt}{}{\varphi_h\in V_h^\partial}{\varphi_h\not\equiv 0}} \frac{(\nabla (y-y_h),\nabla \mathcal E_h \varphi_h)_{L^2(\Omega)^2}}{\|\varphi_h\|_{H^{1/2}(\Gamma)}}
 \end{align}
 with an arbitrary discrete extension operator $\mathcal E_h\colon V_h^\partial\to V_h$.
 In the present case we will use the discrete harmonic extension of $\varphi_h$, this is, $\mathcal E_h = B_h$ which satisfies the estimate $\|\nabla B_h\varphi_h\|_{L^2(\Omega)} \le c\, \|\varphi_h\|_{H^{1/2}(\Gamma)}$. Together with the $H^1(\Omega)$-error estimate from Lemma \ref{lem:error_estimate_l2}
 applied to $\|\nabla(y-y_h)\|_{L^2(\Omega)}$ we conclude the assertion.
\end{proof}
As already mentioned before the previous theorem, we expect a convergence rate higher than one
for convex domains, provided that the input data are more regular. 
The proof of sharp convergence rates in this case is more complicated
and we start with some notation required in the following.
As in \cite{PW17} we introduce a dyadic decomposition towards the boundary of $\Omega$, namely
\begin{equation}\label{eq:def_Omega_J}
 \Omega_J:=\{x\in\Omega\colon \rho(x)\in (d_{J+1},d_J)\}\quad\mbox{for}\ J=-1,\ldots,I,
\end{equation}
where $\rho(x):=\dist(x,\Gamma)$.
We set $d_J:=2^{-J}$ for $J=0,\ldots,I$ and use modifications for the interior domain by $d_{-1}:=\diam(\Omega)$, and the outermost domain by $d_{I+1}:=0$.
Note that this forms a complete decomposition of $\Omega$, i.\,e.,
\begin{equation}\label{eq:dyadic}
 \overline \Omega = \bigcup_{J=-1}^I \overline\Omega_J.
\end{equation}
In the following we will frequently exploit the following two properties that can be directly concluded from the definition:
\begin{equation}\label{eq:properties_dyadic}
 |\Omega_J| \sim d_J,\qquad \inf_{x\in\Omega_J} \dist(x,\Gamma)\sim\sup_{x\in\Omega_J} \dist(x,\Gamma) \sim d_J\ (J\ne I).
\end{equation}
The termination index $I$ is chosen such that $d_I = c_I h$ with some mesh-independent constant $c_I>1$ specified later.
This implies that $I\sim \lvert\ln h\rvert$.
Moreover, we introduce the patches with the adjacent subsets given by
\begin{align*}
  \Omega_J'&:=\Omega_{\min\{I,J+1\}}\cup\overline\Omega_J\cup \Omega_{\max\{-1,J-1\}},\\
  \Omega_J''&:=\Omega_{\min\{I,J+1\}}'\cup\overline\Omega_J\cup \Omega_{\max\{-1,J-1\}}'.
\end{align*}
Note that the patches satisfy the properties \eqref{eq:properties_dyadic}
as well due to $d_{J+1} \sim d_J$ for $J=-1,\ldots,I-1$.

We start the proof of the desired finite element error estimate with some local error estimates for the nodal interpolant defined in \eqref{eq:Ih_tilde}.
\begin{lemma}\label{lem:interpolation_Omega_J}
  Assume that $\Omega$ is convex and
  $y\in  W^{2,\infty}_{\vec\beta}(\Omega)$, $g\in W^{2,2}_{\vec\gamma}(\Gamma)$
  with $\vec\beta\in[0,2)^d$, $\vec\gamma\in [0,3/2)^d$. Then, there holds the estimate
  \begin{align*}
    &\|y-\tilde I_h y\|_{L^2(\Omega_J)} + h\,\|\nabla (y-\tilde I_h y)\|_{H^{1}(\Omega_J)}\\
    &\qquad \le c\, h^{2} d_J^{\min\{1/2,1-\overline{\vec\beta}\}}\,\lvert\ln h\rvert^{z/2}\,|y|_{W^{2,\infty}_{\vec\beta}(\Omega_J')}
  + \delta_{J,I}\,h^{5/2-\overline{\vec\gamma}}\,|g|_{W^{2,2}_{\vec\gamma}(\Gamma)},
  \end{align*}
  with $z=1$ if $\overline{\vec\beta}=1/2$, and $z=0$ if $\overline{\vec\beta}\ne 1/2$.
\end{lemma}
\begin{proof}   
  Throughout the proof we will hide the constant $c_I$ in the generic constant $c$ as
  it is not needed for the terms considered here.
  For elements $T\in\mathcal T_h$ touching a corner, i.\,e., $r_{j,T}=0$ for some $j\in\mathcal C$, we directly deduce the estimate
  \begin{align}\label{eq:int_error_corner}
    &\|y-I_h y\|_{L^2(T)} + h\,\|\nabla(y-I_h y)\|_{L^2(T)} \nonumber\\
    &\quad \le c\, h^{3-\beta_j}\,|y|_{W^{2,\infty}_{\beta_j}(T)}\le c\, h^2\, d_I^{1-\beta_j}\,|y|_{W^{2,\infty}_{\beta_j}(T)},
  \end{align}
  which follows from the estimate from \cite[Corollary 3.33]{Pfe} and
  the property $d_I \sim h$.
  On that part of $\Omega_J$ excluding the elements touching a corner
  we obtain for $J=-1,\ldots,I$ with a standard estimate
  \begin{equation}\label{eq:int_error_H2_reg}
    \|y-I_h y\|_{L^2(\Omega_J\setminus S_h)} + h\,\|\nabla(y-I_h y)\|_{L^2(\Omega_J\setminus S_h)} \le c\, h^2\,\|\nabla^2 y\|_{L^2(\Omega_J'\setminus S_h)},  
  \end{equation}
  where $S_h:=\cup\{T\in\mathcal T_h\colon r_{j,T} = 0,\ j\in\mathcal C\}$.
  It remains to bound the term on the right-hand side of \eqref{eq:int_error_H2_reg}.
  Therefore, we bound  $\|\nabla^2 y\|_{L^2(\Omega_{\tilde J}\setminus S_h)}$
  for $\tilde J=\max\{-1,J-1\},\ldots,\min\{J+1,I\}$ by some weighted $W^{2,\infty}(\Omega)$-norm
  of $y$. This is done by an application of the H\"older inequality on a further dyadic decomposition
  of $\Omega_{\tilde J}$ with respect to the corners.
  A similar technique is used e.\,g.\ in \cite{APW17,Win15} where error estimates in $L^2(\Gamma)$ for the Neumann problem in three-dimensional polyhedral domains are derived. 
  Therein, the domain is decomposed twice into dyadic subsets to resolve both edge and corner singularities.
  Following these ideas we introduce
\begin{equation*}
  d_{J,K}:= 2^K\,d_J = 2^{K-J},
\end{equation*}
and define the subdomains 
\begin{equation}\label{eq:def_parallelogram}
  \Omega_{J,K}^{j,+} := \{x\in\Omega\colon d_{J+1} < \dist(x,\tilde \Gamma_j) \le d_J,\ 
  d_{J,K} < \dist(x,\tilde \Gamma_{j-1}) \le d_{J,K+1}\},
\end{equation}
for $J=0,\ldots,I$, $K=0,\ldots,J-1$ and $j\in\mathcal C$.
Here $\tilde\Gamma_{j}$ stands for
the straight line which coincides with the boundary edge $\Gamma_j$.
Each domain $\Omega_{J,K}^{j,+}$ is a parallelogram bounded by that parallels 
to $\tilde \Gamma_j$ having distance $d_{J+1}$ and $d_{J}$ from $\tilde \Gamma_j$,
and by that parallels to 
$\tilde \Gamma_{j-1}$ having distance $d_{J,K}$ and $d_{J,K+1}$ from $\tilde \Gamma_{j-1}$.
In a similar way we define the subdomains $\Omega_{J,K}^{j,-}$ by simply changing the roles of $\tilde\Gamma_j$ and $\tilde\Gamma_{j-1}$ in the definition \eqref{eq:def_parallelogram}.
Note that $\Omega_{J,0}^j := \Omega_{J,0}^{j,+} = \Omega_{J,0}^{j,-}$.
These subdomains are illustrated in Figure \ref{fig:dyadic}.
\begin{figure}[tb]
  \begin{center}
    \includegraphics[width=.9\textwidth]{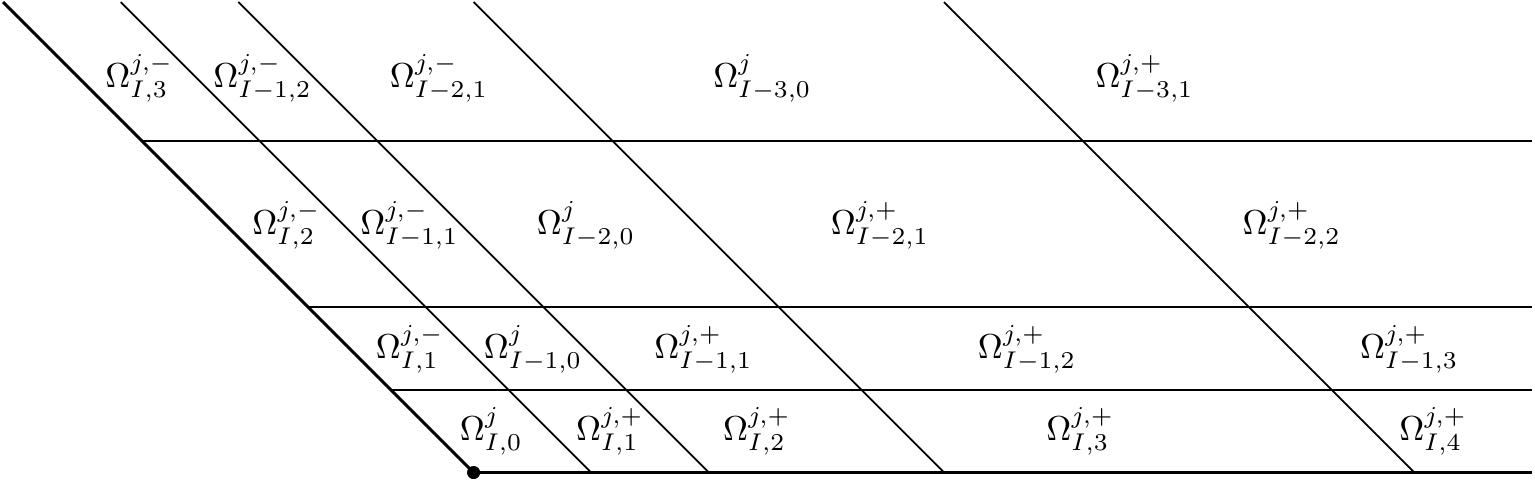}
\end{center}
 \caption{Definition of the domains $\Omega_{J,K}^{j,\pm}$.}
 \label{fig:dyadic}
\end{figure}
By construction we have the property
\begin{equation}\label{eq:volume_Omega_JK}
 |\Omega_{J,K}^{j,\pm}| \sim d_J\, d_{J,K} = d_J^2\,2^K.
\end{equation}
Moreover, we will exploit the property
\begin{equation}\label{eq:prop_distance_corner}
  \inf_{x\in\Omega_{J,K}^{\pm,j}\setminus S_h} r_j(x) \sim \sup_{x\in\Omega_{J,K}^{\pm,j}\setminus S_h} r_j(x)
  \sim d_{J,K},
\end{equation}
for all $J=0,\ldots,I$, $K=0,\ldots,J-1$ and $j\in\mathcal C$,
which follows directly from the definition of the sets $\Omega_{J,K}^{\pm,j}$.
This allows us to locally trade the quantities $d_{J,K}$ by the weights $r_j(x)$ contained
in the weighted Sobolev spaces.
The union of the domains introduced in \eqref{eq:def_parallelogram}
leads to a covering of our initial decomposition \eqref{eq:dyadic} near a ball of radius $1$
around the corner $\boldsymbol c_j$, $j\in\mathcal C$, i.\,e.,
\begin{equation}\label{eq:double_dyadic}
  \Omega_J\cap \Omega_R^j \subset \bigcup_{K=0}^{J-1} \Omega_{J,K}^{j,\pm},\qquad 
  J=0,\ldots,I.
\end{equation}

In order to bound the term on the right-hand side of \eqref{eq:int_error_H2_reg}, we
apply the H\"older inequality on each subset $\Omega_{\tilde J,K}^{j,\pm}$
using the property \eqref{eq:volume_Omega_JK},
and insert appropriate weights taking \eqref{eq:prop_distance_corner}
into account. This implies
\begin{align*}
  \|\nabla^2 y\|_{L^2(\Omega_{\tilde J}\cap\Omega_R^j\setminus S_h)}^2&\le \sum_{K=0}^{\tilde J-1} d_{\tilde J}\,d_{\tilde J,K}^{1-2\beta_j}\,\|r_j^{\beta_j}\,\nabla^2 y\|_{L^\infty(\Omega_{\tilde J,K}^{j,\pm})}^2\\
&\le c\, d_{\tilde J}^{\min\{1,2-2\beta_j\}}\,\lvert\ln h\rvert^z \max_{K=0,\ldots,\tilde J-1} |y|_{W^{2,\infty}_{\beta_j}(\Omega_{\tilde J,K}^{j,\pm})}^2, 
\end{align*}
where the last step follows from the limit value of the geometric series
and the property $\tilde J \le I \sim \lvert\ln h\rvert$, i.\,e., 
\begin{equation}\label{eq:sum_over_K}
  \sum_{K=0}^{\tilde J-1} d_{\tilde J,K}^t = d_{\tilde J}^{t}\sum_{K=0}^{\tilde J-1}(2^K)^t 
  %= \left(\frac{(2^t)^{\tilde J}-1}{2^t-1} \right) 
  \le c\cdot
  \begin{cases}
      d_{\tilde J}^t, &\mbox{if}\ t<0,\\
      1, &\mbox{if}\ t>0,\\
      \lvert\ln h\rvert, &\mbox{if}\ t=0,
      \end{cases}
      \qquad  t:=1-2\beta_j.
    \end{equation}
 With the H\"older inequality we obtain a similar estimate on the set $\Omega_{\tilde J}\setminus\cup_{j\in\mathcal C} \Omega_R^j$, this is,
 \begin{equation*}
   \|\nabla^2 y\|_{L^2(\Omega_{\tilde J}\setminus\cup_{j\in\mathcal C}\Omega_R^j)}^2
   \le c\, d_J\, \|\nabla^2 y\|_{L^\infty(\Omega_{\tilde J}\setminus\cup_{j\in\mathcal C} \Omega_R^j)}^2
   \le c\, d_J\, |y|_{W^{2,\infty}_{\vec\beta}(\Omega_{\tilde J})}^2.
 \end{equation*}
 Combining the previous estimates and summing up
 over the indices
 $\tilde J=\max\{-1,J-1\},\ldots,\min\{J+1,I\}$ finally yields together with \eqref{eq:int_error_corner}
 \begin{align}\label{eq:hoelder_Omega_J}
   &\|y-I_h y\|_{L^2(\Omega_J)} + h\, \|\nabla(y-I_h y)\|_{L^2(\Omega_J)} \nonumber\\
   &\qquad\le c\, h^2\,d_J^{\min\{1/2,1-\overline{\vec\beta}\}}\,\lvert\ln h\rvert^{z/2}\,|y|_{W^{2,\infty}_{\vec\beta}(\Omega_J')}.
 \end{align}
 In case of $J=I$, we still have to discuss the boundary terms to obtain an estimate
 for $\tilde I_h$. This follows from
 \begin{equation*}
   \|\mathcal E_h v_h\|_{L^2(\Omega)} + h\,\|\nabla(\mathcal E_h v_h)\|_{L^2(\Omega)}
   \le c\, h^{1/2}\,\|v_h\|_{L^2(\Gamma)},\qquad v_h\in V_h^\partial,
 \end{equation*}
 and the estimate derived in Lemma \ref{lem:int_error_g}.
\end{proof}

As an intermediate result we prove a weighted $L^2(\Omega)$-error estimate.
The weight function we use is defined by
\begin{equation*}
  \sigma(x) = d_I + \dist(x,\Gamma).
\end{equation*}
Note that such a weight function has been discussed already in Section \ref{sec:weighted_est_bd}.
The regularizer in the present situation is the width of the outermost subset $\Omega_I$.
Here, the relation between the weight function $\sigma$ and the dyadic decomposition \eqref{eq:dyadic}
becomes clear, as the definition directly implies
\begin{equation}\label{eq:trade_sigma_dJ}
  \sigma(x) \sim d_J \qquad \mbox{for}\ x\in\Omega_J,\ J=-1,\ldots,I.
\end{equation}
\begin{lemma}\label{lem:weighted_l2_estimate}
  Assume that $\Omega\subset\mathbb R^2$ is a convex polygonal domain.
  Let $y\in W^{2,\infty}_{\vec\beta}(\Omega)$ and $g\in W^{2,2}_{\vec\gamma}(\Gamma)$
  with $\vec\beta\in [0,2)^d$ and $\vec\gamma\in [0,3/2)^d$.
  Moreover, let $g_h:=Q_h g$. Then the solutions of \eqref{eq:fe_form}
  fulfill the error estimate
 \begin{equation}
   \|\sigma^{-2}\,(y-y_h)\|_{L^2(\Omega)}
   \le c\, \left(h^{\min\{1/2,1-\overline{\vec\beta}\}}\,\lvert\ln h\rvert^{z/2}\,|y|_{W^{2,\infty}_{\vec\beta}(\Omega)}
  + h^{1/2-\overline{\vec\gamma}}\,|g|_{W^{2,2}_{\vec\gamma}(\Gamma)}\right),
 \end{equation}
 provided that $c_I$ is sufficiently large.
\end{lemma}
\begin{proof}
 We follow the arguments of the Nitsche trick using the slightly modified dual problem
 \begin{equation}\label{eq:dual_problem}
  -\Delta w = \sigma^{-2}\,\psi\quad\mbox{in}\quad \Omega,\qquad w=0\quad\mbox{on}\quad\Gamma,
 \end{equation}
 with $\psi = \sigma^{-2}\,(y-y_h)  / \|\sigma^{-2}\,(y-y_h)\|_{L^2(\Omega)}$. Note that $\|\psi\|_{L^2(\Omega)} = 1$. 
 With partial integration and the Galerkin orthogonality we conclude
 \begin{align}\label{eq:nitsche_trick}
   \|\sigma^{-2}\,(y-y_h)\|_{L^2(\Omega)} &= (y-y_h, \sigma^{-2}\,\psi)_{L^2(\Omega)}\nonumber\\
   &=  (\nabla (y-y_h),\nabla (w - I_h w))_{L^2(\Omega)^2} - (g-g_h,\partial_n w)_{L^2(\Gamma)}.
 \end{align}
 
 First, we consider the second term on the right-hand side of \eqref{eq:nitsche_trick}.
 With the orthogonality of the $L^2(\Gamma)$-projection we obtain 
 \begin{equation}\label{eq:boundary_approximation_term}
   (g-g_h,\partial_n w)_{L^2(\Gamma)} \le c\, h^{1/2}\,\|g-I_h g\|_{L^2(\Gamma)}\,\|\partial_n w\|_{H^{1/2}(\Gamma)} \le c\, h^{1/2-\overline{\vec\gamma}}\,|g|_{W^{2,2}_{\vec\gamma}(\Gamma)},   
 \end{equation}
 where the second step follows from Lemma \ref{lem:int_error_g}
 and the estimate $\|\partial_n w\|_{H^{1/2}(\Gamma)} \le c\, \|w\|_{H^2(\Omega)} \le c\, \|\sigma^{-2}\,\psi\|_{L^2(\Gamma)} \le c\, h^{-2}$ which is a consequence of a trace theorem, an a priori estimate, and $\sigma(x) \ge d_I \sim h$ for all $x\in\Omega$.
 
 Next, we discuss the first term on the right-hand side of \eqref{eq:nitsche_trick}.
 A subset-wise application of the Cauchy-Schwarz inequality
 with respect to the dyadic decomposition \eqref{eq:dyadic} yields
 \begin{equation}\label{eq:fe_times_interpolation}
   (\nabla (y-y_h),\nabla (w - I_h w))_{L^2(\Omega)^2}
   \le \sum_{J=-1}^I \|\nabla (y-y_h)\|_{L^2(\Omega_J)}\,\|\nabla (w-I_h w)\|_{L^2(\Omega_J)}.
 \end{equation}
 Moreover, with the local finite-element error estimates from \cite{DGS11} we obtain
 \begin{align}\label{eq:local_fe_error}
   &\|\nabla(y-y_h)\|_{L^2(\Omega_J)}\nonumber\\
   &\quad \le c\, \left(\|\nabla(y-\tilde I_hy)\|_{L^2(\Omega_J')} + d_J^{-1}\,\|y-\tilde I_h y\|_{L^2(\Omega_J')} + d_J^{-1}\,\|y-y_h\|_{L^2(\Omega_J')}\right)
 \end{align}
 for all $J=-1,\ldots,I$. Note that this estimate would not hold for $I_h$ as the boundary
 traces of $y_h$ and the used interpolant must coincide.
  
 Next, we insert \eqref{eq:local_fe_error} into \eqref{eq:fe_times_interpolation} and discuss the resulting terms separately.
 First, consider the product of the interpolation terms.
 For the interpolation error of the dual solution we apply a standard estimate and Lemma \ref{lem:interior_H2} in case of $J=-1,\ldots,I-2$.
 As we can locally trade $\sigma$ by $d_J$, see \eqref{eq:trade_sigma_dJ}, we obtain
 \begin{align}\label{eq:int_error_w}
   \|\nabla(w-I_h w)\|_{L^2(\Omega_J)}
   &\le c\, h\,\|\nabla^2 w\|_{L^2(\Omega_J')}\nonumber\\
   &\le c\, h\,d_J^{-1}\,\left(\|\nabla w\|_{L^2(\Omega_J'')}
     + \|\sigma^{-1}\,\psi\|_{L^2(\Omega_J'')}\right).
 \end{align}
 In case of $J=I-1,I$ we use a global a priori estimate to arrive at
 \begin{equation}\label{eq:w_Ihw_convex}
   \|\nabla(w-I_h w)\|_{L^2(\Omega_J)} \le c\, h\,\|\nabla^2 w\|_{L^2(\Omega)}
   \le c\,h\,d_I^{-1}\,\|\sigma^{-1}\,\psi\|_{L^2(\Omega)},
 \end{equation}
 where the last step follows from the property $\sigma(x) \ge d_I$ for $x\in\Omega$.
 Together with the interpolation error estimates from
 Lemma \ref{lem:interpolation_Omega_J} and the discrete Cauchy-Schwarz inequality we obtain 
 \begin{align}\label{eq:product_interpolation}
   &\sum_{J=-1}^I\left(\|\nabla(y-\tilde I_hy)\|_{L^2(\Omega_J')} + d_J^{-1}\,\|y-\tilde I_h y\|_{L^2(\Omega_J')}\right) \|\nabla(w-I_h w)\|_{L^2(\Omega_J)}\nonumber\\
   &\quad\le c\,h^2\sum_{J=-1}^I  \left(\lvert\ln h\rvert^{z/2}\,d_J^{\min\{-1/2,-\overline{\vec\beta}\}}\,|y|_{W^{2,\infty}_{\vec\beta}(\Omega_J'')} + \delta_{J,I}\,d_I^{-1}h^{1/2-\overline{\vec\gamma}}\,|g|_{W^{2,2}_{\vec\gamma}(\Gamma)}\right)\nonumber\\
   &\qquad \times 
     \left(\|\nabla w\|_{L^2(\Omega_J')} + \|\sigma^{-1}\,\psi\|_{L^2(\Omega_J'')} + (\delta_{J,I-1} + \delta_{J,I})\,\|\sigma^{-1}\,\psi\|_{L^2(\Omega)}\right)\nonumber\\
   &\quad \le c\, h^2 \left(\lvert\ln h\rvert^{z/2}\left(\sum_{J=-1}^I d_J^{2\min\{-1/2,-\overline{\vec\beta}\}}\right)^{1/2}\,|y|_{W^{2,\infty}_{\vec\beta}(\Omega)} + h^{-1/2-\overline{\vec\gamma}}\,|g|_{W^{2,2}_{\vec\gamma}(\Gamma)}\right)\nonumber\\
   &\qquad \times \left(\|\nabla w\|_{L^2(\Omega)} + \|\sigma^{-1}\,\psi\|_{L^2(\Omega)}
     \right)\nonumber\\
   &\quad \le c\, \left(h^{1/2+\min\{0,1/2-\overline{\vec\beta}\}}\,\lvert\ln h\rvert^{z/2}\,|y|_{W^{2,\infty}_{\vec\beta}(\Omega)}
     + h^{1/2-\overline{\vec\gamma}}\,|g|_{W^{2,2}_{\vec\gamma}(\Gamma)}\right).
 \end{align}
 The last step follows from the limit value of the geometric series.
 Analogous to \eqref{eq:sum_over_K} this can be calculated by means of
 \begin{equation}\label{eq:geometric_series}
  \sum_{J=0}^{I-1} d_J^t = \sum_{J=0}^{I-1} (2^{-t})^J \le c\,(1+(2^{-t})^I) \le c\,(1+d_I^t),
 \end{equation}
 with $t=2\min\{-1/2,-\overline{\vec\beta}\}$. Moreover, we exploited the property $d_I\sim h$ 
 and the estimates from Lemma \ref{lem:weighted_h1} taking into account $\|\sigma^{-1}\,\psi\| \le c\,d_I^{-1}\le c\, h^{-1}$. Note that the constant $c_I$ vanishes in $c$ as it is not needed here.
  
 Next, we discuss the product of the pollution term for the primal problem from
 \eqref{eq:local_fe_error} and the interpolation error for the dual problem.
 With similar arguments as in \eqref{eq:product_interpolation} we get
 \begin{align}\label{eq:product_pollution}
  &\sum_{J=-1}^I d_J^{-1}\,\|y-y_h\|_{L^2(\Omega_J')}\,\|\nabla (w-I_h w)\|_{L^2(\Omega_J)} \nonumber\\
   &\quad \le c\, h\,\sum_{J=-1}^I d_J^{-2}\,\|y-y_h\|_{L^2(\Omega_J')}\,\Big(\|\nabla w\|_{L^2(\Omega_J'')}
     + \|\sigma^{-1}\,\psi\|_{L^2(\Omega_J'')} \nonumber\\
   &\hspace{5cm}+ (\delta_{J,I-1} + \delta_{J,I})\,\|\sigma^{-1}\,\psi\|_{L^2(\Omega)}\Big) \nonumber\\
   &\quad \le c\, h\,\|\sigma^{-2}\, (y-y_h)\|_{L^2(\Omega)}\,\left(\|\nabla w\|_{L^2(\Omega)}
     + \|\sigma^{-1}\,\psi\|_{L^2(\Omega)}\right) \nonumber\\
  &\quad \le c\, c_I^{-1}\,\|\sigma^{-2}\,(y-y_h)\|_{L^2(\Omega)},
 \end{align}
 and in the last step we applied Lemma \ref{lem:weighted_h1} and $\|\sigma^{-1}\,\psi\|_{L^2(\Omega)} \le c\, d_I^{-1} = c\, c_I^{-1}\, h^{-1}$.
 Finally, insertion of \eqref{eq:product_interpolation} and \eqref{eq:product_pollution} into \eqref{eq:fe_times_interpolation}
 and the resulting estimate together with \eqref{eq:boundary_approximation_term} into \eqref{eq:nitsche_trick} leads to
 \begin{align}\label{eq:weighted_error_final}
  \|\sigma^{-2}\,(y-y_h)\|_{L^2(\Omega)} &\le c\, h^{1/2 - \max\{0,\overline{\vec\beta}-1/2\}}\,\lvert\ln h\rvert^{z/2}\,|y|_{W^{2,\infty}_{\vec\beta}(\Omega)} \nonumber\\
  &+ c h^{1/2-\overline{\vec\gamma}}\,|g|_{W^{2,2}_{\vec\gamma}(\Gamma)} + c c_I^{-1}\,\|\sigma^{-2}\,(y-y_h)\|_{L^2(\Omega)}.
 \end{align}
 The last term on the right-hand side can be neglected when $c_I$
 is chosen sufficiently large such that $c c_I^{-1} \le 1/2$. This implies the assertion.
\end{proof}
Now, we are in the position to show an improved convergence rate
for the variational normal derivative in case of convex domains.
\begin{theorem}\label{thm:error_estimate}
  Let $g_h:=Q_h g$. Assume that $\Omega$ is a convex polygonal domain. Let $y\in H^2(\Omega)\cap
  W^{2,\infty}_{\vec\beta}(\Omega)$ 
  and $g\in W^{2,2}_{\vec\gamma}(\Gamma)$ with $\vec
  \beta\in [0,1)^d$, $\vec\gamma\in
[0,3/2)^d$. Moreover, it is assumed that $\partial_n y$ is continuous in the
corners of $\Omega$. Then, there holds the error estimate
 \begin{align*}
   &\|\partial_n y - \partial_n^h y_h\|_{H^{-1/2}(\Gamma)} \\
   &\quad\le c\, h^{3/2}\,\lvert\ln h\rvert^{z/2}\,\left(h^{-\max\{0,\overline{\vec\beta}-1/2\}}\,|y|_{W^{2,\infty}_{\vec\beta}(\Omega)}
  + h^{-\overline{\vec\gamma}}\,|g|_{W^{2,2}_{\vec\gamma}(\Gamma)}\right)
 \end{align*}
 with $z=1$ if $\overline{\vec\beta}=1/2$ and $z=0$ otherwise.
\end{theorem}
\begin{proof}
  The beginning of the proof is analogous to the proof of Theorem \ref{thm:main_result_nonconvex}.
  First, we derive an interpolation error estimates for some interpolant of
  $\partial_n y$ in the $L^2(\Gamma)$-norm.
  Therefore, we use the a Cl\'ement-type interpolant $C_h\colon C(\Gamma)\to V_h^\partial$
  with a slight modification in the nodes located in a corner of $\Omega$.
  In the following $\{x_i\}_{i=1}^{N_{\text{bd}}}$ are the nodes of $\mathcal E_h$, and 
  $\{\varphi_i\}_{i=1}^{N_{\text{bd}}}$ are the nodal basis functions of
  $\mathcal E_h$. Each basis function is the boundary trace of a nodal basis
  function of $V_h$ (the 2D ``hat functions'').
  The precise definition of $C_h$ is given by
\begin{equation*}
  [C_h v](x) = \sum_{i=1}^{N_{\text{bd}}} a_i(v)\,\varphi_i(x),\qquad
  a_i(y):=\begin{cases}
    v(x_i),& \mbox{if}\ x_i = \boldsymbol c_j\ \mbox{for some}\ j\in\mathcal C,\\
    |\sigma_i|^{-1}\,\int_{\sigma_i} v, &\mbox{otherwise},
  \end{cases}
\end{equation*}
where $\sigma_i :=\cup\{E\in\mathcal E_h\colon x_i\in \overline E\}$ if $x_i\not\in\{\boldsymbol c_j,\ j\in\mathcal C\}$.
For the nodes $x_i$ located in the vertices of $\Omega$ we simply set $\sigma_i:=\emptyset$.
For some $E\in\mathcal E_h$ we denote by $T$ the corresponding
triangle from $\mathcal T_h$, this is, $E\subset \bar T$,
and by $F_T\colon \hat T\to T$ the affine mappings from the reference triangle
$\hat T:=\text{conv}\{(0,0),\, (1,0),\,(0,1)\}$ to the world element $T$.
Moreover, we will use the notation $\hat v(\hat x) := v(F_T(\hat x))$.
In addition, we introduce the patches
$S_E:=\cup\{\sigma_i\colon x_i\in\bar E\}$ and  $D_E:=\cup\{T\in\mathcal
T_h\colon \bar T\cap \bar E\ne\emptyset\}$,
as well as  the corresponding reference patches $S_{\hat E} := F_T^{-1}(S_E)$ and $D_{\hat E} := F_T^{-1}(D_E)$.

First, we easily see that the interpolant satisfies the stability estimate
\begin{equation*}
  \|C_h (v)\|_{L^2(E)}
  \le \sum_{i\colon x_i\in\bar E} a_i(v)\,\|\varphi_i\|_{L^2(E)}
  \le c\, |E|^{1/2}\,\|v\|_{L^\infty(S_E)} 
\end{equation*}
for an arbitrary function $v\in L^\infty(E)$.
For elements $E\in\mathcal E_h$ touching the corner $\boldsymbol c_j$,
$j\in\mathcal C$, we insert an arbitrary first-order polynomial $p$ and infer with 
the triangle inequality and the stability estimate for $C_h$ 
\begin{align*}
  \|\partial_n y- C_h (\partial_n y)\|_{L^2(E)}
  &\le c\, \left(\|\partial_n y - \partial_n p \|_{L^2(E)} + \|C_h (\partial_n y - \partial_n p)\|_{L^2(E)}\right) \\
  &\le c\, h^{-1}\,|E|^{1/2}\,\|\partial_{\hat n}\hat y - \partial_{\hat n} \hat p\|_{L^\infty(S_{\hat E})} \\
  &\le c\, h^{-1}\,|E|^{1/2}\,\|\hat y - \hat p\|_{W^{1,\infty}(D_{\hat E})}.
\end{align*}
We proceed with the embedding $W^{2,2+\varepsilon}(D_{\hat E}) \hookrightarrow W^{1,\infty}(D_{\hat E})$,
the Bramble-Hilbert Lemma, as well as the embedding 
$W^{0,\infty}_{\beta_j}(D_{\hat E}) \hookrightarrow L^{2+\varepsilon}(D_{\hat E})$, which holds for all $\beta_j < 1$, provided that $\varepsilon>0$ is sufficiently small.
The weighted Sobolev spaces in the reference setting are defined analogous to the
spaces defined in Section \ref{sec:weighted_spaces} with the exception that
the weight function is defined by $\hat r:=|\hat x|$.
When assuming w.l.o.g\ that $F_T(0) = \boldsymbol c_j$
we obtain the property $\hat r(\hat x) \sim r_j(F_T(\hat x))\, h^{-1}$.
A transformation of variables then yields
\begin{equation}\label{eq:int_error_corner_convex}
  \|\partial_n y- C_h (\partial_n y)\|_{L^2(E)}
  \le c\, h^{-1}\,|E|^{1/2}\,|\hat y|_{W^{2,\infty}_{\beta_j}(D_{\hat E})}
  \le c\, h^{3/2-\beta_j}\,|y|_{W^{2,\infty}_{\beta_j}(D_E)}.
\end{equation}
For elements $E\in\mathcal E_h$ away from the corners we apply similar arguments,
but use instead the stability estimate $\|C_h v\|_{L^2(E)} \le c\,
\|v\|_{L^2(S_E)}$, to arrive at
\begin{align*}
  \|\partial_n y- C_h (\partial_n y)\|_{L^2(E)} &\le c\,  h^{-1}\,|E|^{1/2}\,\|\partial_{\hat n} \hat y - \partial_{\hat n} \hat p\|_{L^2(S_{\hat E})} \\
&\le c\, h^{-1}\,|E|^{1/2}\,\|\hat y -\hat p\|_{H^2(D_{\hat E})} \le c\, h^{-1}\,|E|^{1/2}\,|\hat y|_{H^2(D_{\hat E})} \\
&\le c\, h^{1/2}\,|y|_{H^2(D_E)}.
\end{align*}
Summation over all boundary elements $E\in\mathcal E_h$ with $E\not\subset
S_h:=\cup\{E\in\mathcal E_h\colon r_{j,E} = 0,\ j\in\mathcal C\}$ yields
\begin{align*}
  \|\partial_n y - C_h(\partial_n y)\|_{L^2(\Gamma\setminus \overline S_h)} 
&\le c\, h^{1/2}\,|y|_{H^2(\Omega_I\setminus S_h)} \\
&\le c\, h^{1 - \max\{0,\overline{\vec\beta}-1/2\}}\,\lvert\ln h\rvert^{z/2}\,|y|_{W^{2,\infty}_{\vec\beta}(\Omega)},
\end{align*}
where the last step is an application of the estimate \eqref{eq:hoelder_Omega_J}
with $J=I$ taking into account $d_I\sim h$.
Together with the estimates \eqref{eq:int_error_duality} and \eqref{eq:int_error_corner_convex}
we deduce
\begin{equation*}
  \|\partial_n y - Q_h(\partial_n y)\|_{H^{-1/2}(\Gamma)}
  \le c\, h^{3/2-\max\{0,\overline{\vec\beta}-1/2\}}\,\lvert\ln h\rvert^{z/2}\,|y|_{W^{2,\infty}_{\vec\beta}(\Omega)}.
\end{equation*}

The fully discrete part $Q_h(\partial_ny) - \partial_n^h y_h$
is treated with \eqref{eq:discrete_part}. However, now, 
the extension operator we are going to use is the Lagrange interpolant
of the harmonic extension, this is, $\mathcal E_h := I_h B$.
Introducing $B\varphi_h$ as intermediate function yields
\begin{align}\label{eq:h-12_norm_splitting}
  &(\nabla (y-y_h),\nabla \mathcal E_h \varphi_h)_{L^2(\Omega)^2}\nonumber\\
  &\qquad = (\nabla (y-y_h),\nabla (I_h B \varphi_h - B\varphi_h))_{L^2(\Omega)^2} + (\nabla (y-y_h),\nabla B \varphi_h)_{L^2(\Omega)^2}.
\end{align}
The latter term is the simpler one.
By partial integration, the trace theorem for normal derivatives from
 \cite[Theorem 1.3.2]{Mik12} and the interpolation error estimate
 from Lemma \ref{lem:int_error_g} we obtain 
 \begin{align}\label{eq:main_term_1}
   (\nabla (y-y_h),\nabla B \varphi_h)_{L^2(\Omega)^2}
   &= (y-y_h,\partial_n B\varphi_h)_{L^2(\Gamma)} \nonumber\\
   & \le c\, \|g-Q_h g\|_{H^{1/2}(\Gamma)}\,\|B\varphi_h\|_{H^1(\Omega)} \nonumber\\
   & \le c\, h^{3/2-\overline{\vec\gamma}}\,|g|_{W^{2,2}_{\vec\gamma}(\Gamma)}\,\|\varphi_h\|_{H^{1/2}(\Gamma)}.
 \end{align}
 
 The first term on the right-hand side of \eqref{eq:h-12_norm_splitting}
 has the structure of the term
 \eqref{eq:fe_times_interpolation} from the proof of Lemma \ref{lem:weighted_l2_estimate}. 
 The only difference is that the dual solution $w$ used in this lemma, has to be
 replaced by the function $B\varphi_h$. This means that $w$ fulfills a 
 homogeneous equation and but inhomogeneous Dirichlet boundary conditions in the present case.
 For that reason, we replace $\psi$ by $0$ and it remains 
 to bound each occurrence of $w$ by the $H^{1/2}(\Gamma)$-norm of $\varphi_h$.

 To be more precise, we confirm that the estimate \eqref{eq:int_error_w} remains valid when neglecting
 the term depending on $\psi$. 
 Moreover, we have to establish an analogue to the estimate
 \eqref{eq:w_Ihw_convex}.
 For the present definition of $w:=B \varphi_h$ we obtain 
 with a standard interpolation and a priori estimate and an inverse inequality
 \begin{align*}
   \|\nabla(w - I_h w)\|_{H^1(\Omega)}
   &\le c\, h^{1/2-\varepsilon}\,\|w\|_{H^{3/2-\varepsilon}(\Omega)}\\
   &\le c\, h^{1/2-\varepsilon}\,\|\varphi_h\|_{H^{1-\varepsilon}(\Gamma)}
   \le c\, \|\varphi_h\|_{H^{1/2}(\Gamma)},
 \end{align*}
 provided that $\varepsilon\in(0,1/2)$.
 With this modification we can repeat the arguments used to show \eqref{eq:weighted_error_final}.
 Moreover, we have to modify the last steps in \eqref{eq:product_interpolation}
 and \eqref{eq:product_pollution}. There, we insert the a priori estimate
 $\|\nabla w\|_{L^2(\Omega)} \le c\,\|\varphi_h\|_{H^{1/2}(\Gamma)}$.
 In both estimates the exponent of $h$ is then greater by one.
 
 All together, this implies
 \begin{align*}
   &(\nabla (y-y_h),\nabla (I_h B \varphi_h -
   B\varphi_h))_{L^2(\Omega)^2} \\
   &\quad \le c\, h^{3/2}\Big(h^{\min\{0,1/2-\overline{\vec\beta}\}}\,\lvert\ln h\rvert^{z/2}\,|y|_{W^{2,\infty}_{\vec\beta}(\Omega)} 
     + h^{-\overline{\vec\gamma}}\,|g|_{W^{2,2}_{\vec\gamma}(\Gamma)}\\
   &\hspace{1.4cm}+ h^{-1/2}\,\|\sigma^{-2}\,(y-y_h)\|_{L^2(\Omega)}\Big)\,\|\varphi_h\|_{H^{1/2}(\Gamma)}.
 \end{align*}
 The last term in the parentheses on the right-hand side is discussed in Lemma
 \ref{lem:weighted_l2_estimate} already and we can bound this term by the
 first two ones.
 
 Insertion of the previous estimate, \eqref{eq:main_term_1} and \eqref{eq:h-12_norm_splitting}
 into \eqref{eq:discrete_part} and canceling out the terms $\|\varphi_h\|_{H^{1/2}(\Gamma)}$
 leads to the desired estimate for the term
 $\|Q_h(\partial_n y) - \partial_n^h y_h\|_{H^{-1/2}(\Gamma)}$.
\end{proof}

\begin{remark}
  The best possible convergence rate of $3/2$ is achieved when $g\in H^2(\Gamma)$
  and $y\in W^{2,\infty}_{\vec\beta}(\Omega)$ with $\beta_j < 1/2$ for all $j\in\mathcal C$.
  In general, the latter assumption is only satisfied when the opening angles of the corners
  of $\Omega$ satisfy $\omega_j < 2\pi/3$, $j\in\mathcal C$,
  and when $f$ is sufficiently smooth. As an example, assuming $f$ to be H\"older continuous
  would be sufficient, compare Corollary \ref{cor:reg_W2inf}.
  Otherwise, for angles larger than $2\pi/3$ we find a relation between
  the convergence rate and the exponent of the dominating singularity
  $\bar \lambda = \pi/\omega_{\max}$ by choosing 
  $\bar\beta = 2-\bar\lambda + \varepsilon$ if $\omega_{\max}\in (2\pi/3,\pi)$ and
  $\bar\alpha = 1-\bar\lambda+\varepsilon$ if $\omega_{\max}\in (\pi,2\pi)$ for arbitrary but
  sufficiently small $\varepsilon > 0$.
  Under the assumption that $f$ and $g$ are sufficiently smooth we then infer
  \begin{equation*}
    \|\partial_n y - \partial_n^h y_h\|_{H^{-1/2}(\Gamma)} \le c\, h^{\min\{3/2,\bar\lambda-\varepsilon\}}\,\lvert\ln h\rvert^{z/2}.
  \end{equation*}
\end{remark}

\section{Dirichlet control problems}\label{sec:control}

This section is devoted to the numerical approximation of the optimal control problem
\begin{equation}\label{eq:target}
 J(u,z):=\frac12\,\|u-u_d\|_{L^2(\Omega)}^2 + \frac\nu2\left<Nz,z\right> \to\min!
\end{equation}
subject to the constraints
\begin{equation}\label{eq:state_eq}
\left\lbrace
 \begin{aligned}
  -\Delta u &= f &\qquad & \mbox{in}\ \Omega\\
  u&= z && \mbox{on}\ \Gamma.
 \end{aligned}
 \right.
\end{equation}
Here, $f, u_d\in L^2(\Omega)$ are given functions and $\nu>0$ is a
regularization parameter. The operator $N\colon H^{1/2}(\Gamma)\to
H^{-1/2}(\Gamma)$ is a Steklov-Poincar\'e operator which is used to realize
an $H^{1/2}(\Gamma)$-seminorm. 

We introduce the linear operators $S\colon H^{1/2}(\Gamma)\to H^1(\Omega)$ and
$P\colon H^{1}(\Omega)^*\to H_0^1(\Omega)$ defined by
\begin{align*}
  u_z = Sz\ &:\iff   u_z\ \mbox{solves \eqref{eq:state_eq} for } f\equiv 0,\\
  u_f = Pf\ &:\iff   u_f\ \mbox{solves \eqref{eq:state_eq} for } z\equiv 0.
\end{align*}
We can express the operator $N$ by means of 
$Nz:=\partial_n (Sz)$.
Note that the regularization term is equivalent to the square of
the $H^{1/2}(\Gamma)$-seminorm of $z$.

Necessary optimality conditions, that are also sufficient due to the convexity of this optimization problem, can be found in \cite{OPS13}.
Therein, it is shown that the pair $(u,z)\in H^1(\Omega)\times H^{1/2}(\Gamma)$ is the unique global minimizer of \eqref{eq:target}--\eqref{eq:state_eq} if and only if
an adjoint state $p\in H^1(\Omega)$ exists such that the optimality system
\begin{equation}\label{eq:optimality_system}
  \left\lbrace
\begin{aligned}
 - \Delta u &= f& -\Delta p &= u-u_d &\quad& \mbox{in}\ \Omega,\\
 u&= z & p&= 0 && \mbox{on}\ \Gamma,\\
  & & \nu\,N z + \partial_n p &= 0 && \mbox{in}\ H^{-1/2}(\Gamma),
\end{aligned}
\right.
\end{equation}
is fulfilled. 
One can reformulate the optimality system using the operators $S$ and $P$
introduced above. Taking also into account the relation 
$S^* u=\partial_n Pu $ leads to a compact form of the optimality system
\begin{equation*}
 u=Sz+Pf,\qquad \nu\,Nz + S^*(u-u_d) = 0.
\end{equation*}
Eliminating $u$ leads to the variational problem
\begin{equation}\label{eq:variational_form}
 \left<T^\nu z,v\right> = \left<g,v\right>\qquad\forall v\in H^{1/2}(\Gamma)
\end{equation}
with
\begin{equation*}
T^\nu:= S^*S + \nu N,\qquad g:=S^*(u_d- u_f),
\end{equation*}
where $u_f:=Pf$. 
The existence of a unique solution $z$ of \eqref{eq:variational_form} follows
from the Lax-Milgram Lemma.
It remains to discuss the regularity of the optimal solution and the
corresponding state and adjoint state.
These results will be needed for sharp discretization error estimates.
\begin{lemma}\label{lem:regularity}
  Assume that $f,u_d\in L^2(\Omega)$.
  Let $\vec\alpha\in [0,1)^d$  
  be a weight vector satisfying
  $1-\lambda_j< \alpha_j$, $j\in\mathcal C$.
  Then, the solution of \eqref{eq:optimality_system} possesses the regularity
  \begin{equation}\label{eq:weighted_h2_reg}
    Sz, Pf\in W^{2,2}_{\vec\alpha}(\Omega),\quad
    p\in V^{2,2}_{\vec\alpha}(\Omega),\quad
    z\in W^{3/2,2}_{\vec\alpha}(\Gamma),\quad
  \end{equation}
  Moreover, if $u_d\in C^{0,\sigma}(\overline\Omega)$ for some $\sigma\in (0,1)$, there holds
  \begin{equation}\label{eq:weighted_w2inf_reg}
    Sz\in W^{2,\infty}_{\vec\beta}(\Omega),\quad p\in V^{2,\infty}_{\vec\beta}(\Omega),\quad z\in W^{2,\infty}_{\vec\beta}(\Gamma)\cap W^{2,2}_{\vec\gamma}(\Gamma),
  \end{equation}
  with $\vec\beta\in [0,2)^d$, $\vec\gamma\in [0,3/2)^d$ satisfying $2-\lambda_j < \beta_j$
  and $3/2-\lambda_j<\gamma_j$ for $j\in\mathcal C$.
\end{lemma}
\begin{proof}
  In order to transfer the regularity of the adjoint state to the state, we introduce the auxiliary function $u_0$ solving the boundary value problem
  \begin{equation*}
    -\Delta u_0 = \frac1\nu\,(u-u_d) \quad \mbox{in}\ \Omega,\qquad
    \partial_n u_0 = 0\quad \mbox{on}\ \Gamma.
  \end{equation*}
  Note that the function $u_0$ can be determined uniquely as the optimal state satisfies  
  $\int_\Omega u = \int_\Omega u_d$, see e.\,g.\ \cite[Section 3.2.3]{John14}.
  With \eqref{eq:optimality_system} it is easy to confirm that the state can be
  decomposed by means of $Sz = u_0 - \frac1\nu\,p$.
  
  The assertion then follows from bootstrapping arguments.       
  Standard regularity results, and in particular \cite[Theorem 4.4.3.7]{Gri85}, immediately imply
  \begin{align*}
    z\in H^{1/2}(\Gamma)\ \Rightarrow\ &u\in H^1(\Omega)\hookrightarrow L^q(\Omega)
    \ \Rightarrow\ p, u_0\in W^{2,q}(\Omega) \\
    \Rightarrow\ &Sz\in W^{2,q}(\Omega)\ \Rightarrow\ z\in W^{2-1/q,q}(\Gamma),
  \end{align*}
  for arbitrary $q\in[1,\infty)$ satisfying $2/q > 2-\lambda_j$ for all $j\in\mathcal C$.
  The regularity results collected in \eqref{eq:weighted_h2_reg}
  then directly follow from Lemma \ref{lem:regularity_W22}.
  
  From embedding theorems we moreover conclude that $u, p\in C^{0,\sigma'}(\overline\Omega)$ for
  some $\sigma' \in (0,\min\{1,\bar\lambda\})$, and 
  with Corollary \ref{cor:reg_W2inf} we directly infer
  \eqref{eq:weighted_w2inf_reg}. The assertion $z\in W^{2,2}_{\vec\gamma}(\Gamma)$ follows from
  the $W^{2,\infty}_{\vec\beta}(\Gamma)$-regularity due to the H\"older inequality.
\end{proof}

In order to discretize the optimality system we 
replace $S$ and $P$ by the finite element solution 
operators $S_h\colon V_h^\partial\to V_h$ and $P_h\colon H^1(\Omega)^*\to V_h^\partial$ defined by
\begin{align*}
 u_h&= S_h z_h&&:\iff& u_h|_\Gamma\equiv z_h\quad (\nabla u_h,\nabla v_h)_{L^2(\Omega)^2} &= 0 &\forall v_h\in V_{0h},\\
 p_h&= P_h u_h&&:\iff& (\nabla p_h,\nabla v_h)_{L^2(\Omega)^2} &= (u_h,v_h)_{L^2(\Omega)} &\forall v_h\in V_{0h}.
\end{align*}
Instead of $S^*$ we use its discrete version 
$S_h^*:=\partial_n^h P_h$ which is the adjoint operator to $S_h$.
Then, we seek a state $u_h\in V_h$ and a control $z_h\in V_h^\partial$  as solution of the finite-dimensional optimization problem
\begin{equation}\label{eq:target_discrete}
 J_h(u_h,z_h):=\frac12\,\|u_h - u_d\|_{L^2(\Omega)}^2 + \frac\nu2\left<N_h z_h,z_h\right> \to \min!
\end{equation}
subject to 
\begin{equation}\label{eq:state_eq_discrete}
 u_h|_\Gamma\equiv z_h,\quad (\nabla u_h,\nabla v_h)_{L^2(\Omega)^2} = (f,v_h)_{L^2(\Omega)}\qquad\forall v_h\in V_{0h}.
\end{equation}
In order to define an appropriate discrete Steklov-Poincar\'e operator 
we use the variational normal derivative introduced in \eqref{eq:discrete_normal_deriv},
and define $N_h\colon V_h^\partial\to V_h^\partial$ by $N_h z_h := \partial_n^h (S_h z_h)$.
Note that by this definition, the functional $\left<N_h\cdot, \cdot\right>$
induces a mesh-independent $H^{1/2}(\Gamma)$-seminorm for functions in $V_h^\partial$.
Analogous to the continuous case we can derive the discrete optimality system
\begin{align}\label{eq:optimality_system_disc}
 u_h|_\Gamma\equiv z_h, \quad (\nabla u_h,\nabla v_h)_{L^2(\Omega)^2} &= (f,v_h)_{L^2(\Omega)} && \forall v_h\in V_{0h},\nonumber\\
 (\nabla p_h,\nabla v_h)_{L^2(\Omega)^2} &= (u_h-u_d,v_h)_{L^2(\Omega)} && \forall v_h\in V_{0h},\\
 (\nu\,\partial_n^h (S_h z_h) + \partial_n^h p_h, w_h)_{L^2(\Gamma)}&=0 && \forall w_h\in V_h^\partial,\nonumber
\end{align}
with an adjoint state $p_h\in V_{0h}$.
This system can be rewritten by means of
\begin{equation}\label{eq:variational_form_discrete}
 \left<T_h^\nu z_h,v_h\right> = \left<g_h,v_h\right>\qquad\forall v_h\in V_h^\partial
\end{equation}
with $T_h^\nu = S_h^*S_h + \nu\, N_h$, $g_h:=S_h^*(u_d - u_{f,h})$ and $u_{f,h}:=P_h f$.

The remainder of this section is devoted to the proof of error estimates for the finite-element approximation $(u_h,z_h,p_h)$.
To this end, we introduce an auxiliary function $\tilde z_h\in V_h^\partial$ solving the variational formulation
\begin{equation}\label{eq:def_tilde_zh}
\left<T^\nu \tilde z_h,v_h\right> = \left<g,v_h\right>\qquad\forall v_h\in V_h^\partial.
\end{equation}
The Lax-Milgram-Lemma guarantees the existence and uniqueness of $\tilde z_h$ and 
by the C\'ea-Lemma and the
interpolation error estimates from Lemma \ref{lem:int_error_g} we obtain the following intermediate result:
\begin{lemma}\label{lem:cea}
  Let $z\in H^{1/2}(\Gamma)$ be the optimal control solving \eqref{eq:target}--\eqref{eq:state_eq}.
  The approximate solutions $\tilde z_h$ of \eqref{eq:def_tilde_zh} satisfy the estimate
 \begin{equation*}
   \|z-\tilde z_h\|_{H^{1/2}(\Gamma)} \le c\,
   \begin{cases}
     h^{1-\overline{\vec\alpha}}\,\|z\|_{W^{3/2,2}_{\vec\alpha}(\Gamma)}, &\mbox{if}\ u_d\in L^2(\Omega),\\
     h^{3/2-\overline{\vec\gamma}}\,|z|_{W^{2,2}_{\vec\gamma}(\Gamma)}, &\mbox{if}\ u_d\in C^{0,\sigma}(\overline\Omega).
   \end{cases}
 \end{equation*}
 The weights $\vec\alpha$ and $\vec\gamma$ are chosen as in Lemma \ref{lem:regularity}.
\end{lemma}
It remains to derive an estimate for the error between the continuous and the discrete control
$z$ and $z_h$, respectively. In the following Lemma we present a general
estimate. The idea of the proof is taken from \cite{OPS13}.
\begin{lemma}
  The solutions $z$ and $z_h$ of \eqref{eq:variational_form} 
  and \eqref{eq:variational_form_discrete}, respectively, satisfy the general error estimate
 \begin{align*}
  \|z-z_h\|_{H^{1/2}(\Gamma)}
  &\le c\,\big(\|z-\tilde z_h\|_{H^{1/2}(\Gamma)} + \|\partial_n p - \partial_n^h p_h(u)\|_{H^{-1/2}(\Gamma)} \\
  &\quad+ \|u - u_h(Q_h z)\|_{L^{2}(\Omega)} + \|\partial_n(Sz) - \partial_n^h(S_hQ_h z)\|_{H^{-1/2}(\Gamma)}\big),
 \end{align*}
 with $u_h(Q_h z) = S_h(Q_h z) + u_{f,h}$ and $p_h(u) = P_h(u-u_d)$.
\end{lemma}
\begin{proof}
  First, we confirm that the bilinear form $\left<T_h^\nu \cdot, \cdot\right>$
  is $V_h^\partial$-elliptic and continuous, this is, for all $v_h,w_h\in V_h^\partial$ there holds
  \begin{align*}
    \gamma\, \left<T_h^\nu v_h,v_h\right> &\ge \|v_h\|_{H^{1/2}(\Gamma)}^2,\\
    \left<T_h^\nu v_h,w_h\right>&\le c\, \|v_h\|_{H^{1/2}(\Gamma)}\,\|w_h\|_{H^{1/2}(\Gamma)},
  \end{align*}
  with some constant $\gamma>0$ independent of $h$.
  This follows directly from the mapping properties
  of $N_h$, $S_h$ and $S_h^*$ as well as Lemma \ref{lem:stability_normal}. 
  In the following we write $w_h:= z_h - \tilde z_h$.
  With the ellipticity, the equations \eqref{eq:variational_form_discrete}  
  and \eqref{eq:def_tilde_zh} and Young's inequality, we obtain
 \begin{align}\label{eq:general_est_1}
  \gamma\,\|w_h\|_{H^{1/2}(\Gamma)}^2 &\le \left<T^\nu_h(z_h-\tilde z_h),w_h\right> \nonumber\\
  &= \left<g_h - g + (T^\nu - T_h^\nu)\tilde z_h,w_h\right> \nonumber\\
  &\le \frac{\gamma}2\,\|w_h\|_{H^{1/2}(\Gamma)}^2 
  + c\,\|g-g_h+(T^\nu - T_h^\nu) \tilde z_h\|_{H^{-1/2}(\Gamma)}^2.  
 \end{align}
 Insertion of the definitions of $g$ and $g_h$ yields
 \begin{equation}\label{eq:general_est_2}
  g_h-g = (S^*-S_h^*)(Pf-u_d) + S_h^* (P-P_h) f.
 \end{equation}
 Rearrangement of the remaining terms and the definitions of $T^\nu$ and $T_h^\nu$ lead to
 \begin{align}\label{eq:general_est_3}
 & (T^\nu - T_h^\nu) \tilde z_h\nonumber\\
 &\quad=  T^\nu(\tilde z_h- z) + T^\nu z - T_h^\nu Q_h z + T_h^\nu (Q_h z - \tilde z_h) \nonumber\\
   &\quad= T^\nu(\tilde z_h - z) + (S^*-S_h^*)Sz + S_h^*(Sz - S_h Q_h z) \nonumber\\
   &\qquad + \nu\,(Nz - N_h Q_h z) + T_h^\nu(Q_h z - \tilde z_h)
 \end{align}
 Next, we insert \eqref{eq:general_est_2} and \eqref{eq:general_est_3} into \eqref{eq:general_est_1},
 apply the triangle inequality, and use the abbreviations 
 \[u = Sz + Pf,\quad \partial_n p = S^*(u-u_d),\quad \partial_n(Sz) = Nz,\]
 as well as their discrete counterparts
 \[u_h(Q_h z) = S_hQ_h z + P_h f,\quad \partial_n^h p_h(u)= S_h^*(u-u_d),\quad \partial_n^h(S_hQ_h z) = N_h Q_h z.\]
 Insertion of \eqref{eq:general_est_2} and \eqref{eq:general_est_3} into
 \eqref{eq:general_est_1}, and exploiting the stability estimates
 \begin{align*}
   \|T^\nu v\|_{H^{-1/2}(\Gamma)}&\le c\, \|v\|_{H^{1/2}(\Gamma)}, &
   \|T_h^\nu v_h\|_{H^{-1/2}(\Gamma)} &\le c\, \|v_h\|_{H^{1/2}(\Gamma)},\\
   \|S_h^* v\|_{H^{-1/2}(\Gamma)} &\le c\, \|v\|_{L^2(\Omega)}, &&
 \end{align*}
  that can be concluded from Lemma \ref{lem:stability_normal}, as well as the stability of $Q_h$ in $H^{1/2}(\Gamma)$ \cite{Ste01}, leads to the estimate
 \begin{align*}
  \frac\gamma2\|w_h\|_{H^{1/2}(\Gamma)}^2 &\le c\,\|z-\tilde z_h\|_{H^{1/2}(\Gamma)}^2 + \|\partial_n p - \partial_n^h p_h(u)\|_{H^{-1/2}(\Gamma)}^2 \\
  &+ c\,\|u - u_h(Q_h z)\|_{L^{2}(\Omega)}^2
  + \|\partial_n(Sz) - \partial_n^h(S_hQ_h z)\|_{H^{-1/2}(\Gamma)}^2.  
 \end{align*}  
 With the triangle inequality $\|z-z_h\|_{H^{1/2}(\Gamma)} \le \|z-\tilde z_h\|_{H^{1/2}(\Gamma)} + \|w_h\|_{H^{1/2}(\Gamma)}$ we conclude the
 assertion.
\end{proof}

This general estimate and the estimates presented in Lemma \ref{lem:error_estimate_l2},
Theorems \ref{thm:main_result_nonconvex}, \ref{thm:error_estimate} and Lemma \ref{lem:cea} lead to the main result of this section.
\begin{theorem}\label{thm:ocp_h12_error}
  Let $\Omega\subset\mathbb R^2$ be an arbitrary polygonal domain and assume that
  $f, u_d\in L^2(\Omega)$.
  Let $(u,z,p)$ be the solution of \eqref{eq:optimality_system}, and $(u_h,z_h,p_h)$ the
  corresponding finite element approximation solving \eqref{eq:optimality_system_disc}.
  Then, the error estimate 
  \begin{equation}\label{eq:main_result_arbitrary}
    \|z-z_h\|_{H^{1/2}(\Gamma)} 
    \le c\, h^{\min\{1,\overline\lambda-\varepsilon\}}
  \end{equation}
  is valid for arbitrary $\varepsilon>0$. 
  
  Furthermore, if $\Omega$ is convex and $u_d\in C^{0,\sigma}(\overline\Omega)$ for some $\sigma\in(0,1)$,
  there holds the estimate
  \begin{equation}\label{eq:main_result_convex}
    \|z-z_h\|_{H^{1/2}(\Gamma)} 
    \le c\, h^{\min\{3/2,\overline\lambda-\varepsilon\}}.
  \end{equation}  
  Note that $\bar\lambda := \pi/\max_{j\in\mathcal C} \omega_j$.
  
  The constant $c$ depends linearly on the functions $z, Sz, Pf$ and $p$, more precisely,
  \[
    c = \begin{cases}
      c\left(\|z\|_{W^{3/2,2}_{\vec\alpha}(\Gamma)}
        +|Sz|_{W^{2,2}_{\vec\alpha}(\Omega)} + |Pf|_{W^{2,2}_{\vec\alpha}(\Omega)} + |p|_{W^{2,2}_{\vec\alpha}(\Omega)}\right), & \mbox{in }\eqref{eq:main_result_arbitrary},\\
      c\left(|z|_{W^{2,2}_{\vec\gamma}(\Gamma)}
        +|Sz|_{W^{2,\infty}_{\vec\beta}(\Omega)} +|Pf|_{W^{2,2}_{\vec\alpha}(\Omega)} +  |p|_{W^{2,\infty}_{\vec\beta}(\Omega)}\right), & \mbox{in }\eqref{eq:main_result_convex}.   
    \end{cases}
  \]
  The weights are defined by $\alpha_j:=\max\{0,1-\lambda_j+\varepsilon\}$, 
  $\beta_j:=\max\{0,2-\lambda_j+\varepsilon\}$ and $\gamma_j:=\max\{0,3/2-\lambda_j+\varepsilon\}$ for all $j\in\mathcal C$.
\end{theorem}
As a simple conclusion we also obtain an error estimate for the state variable in the energy norm.
\begin{corollary}\label{cor:estimate_states}
  Assume that $f,\,u_d\in L^2(\Omega)$.
  Let $u\in H^1(\Omega)$ and $u_h\in V_h$ be the optimal states of
  \eqref{eq:target}-\eqref{eq:state_eq} and   \eqref{eq:target_discrete}-\eqref{eq:state_eq_discrete},
  respectively.
  Then, the error estimate
  \begin{equation*}
    \|u-u_h\|_{H^1(\Omega)} \le c\, h^{\min\{1,\overline\lambda-\varepsilon\}}
  \end{equation*}
  holds for arbitrary but sufficiently small $\varepsilon>0$.
  The constant $c>0$ is the same as in the previous theorem.
\end{corollary}
\begin{proof}
  With the triangle inequality we get
  \begin{equation*}
    \|u-u_h\|_{H^1(\Omega)}
    \le \|Sz+u_f - (S_h Q_h z+u_{f,h})\|_{H^1(\Omega)} + \|S_hQ_h(z-z_h)\|_{H^1(\Omega)}.
  \end{equation*}
  Note that $S_h Q_h z + u_{f,h}$ is the finite element approximation of $u:=Sz + u_f$.
  Thus, we infer with Lemma~\ref{lem:error_estimate_l2}
  \begin{equation*}
    \|Sz+u_f - (S_h Q_h z+u_{f,h})\|_{H^1(\Omega)}
    \le c\, h^{\min\{1,\overline\lambda-\varepsilon\}}\,\left(|u|_{W^{2,2}_{\vec\alpha}(\Omega)}
        + \|z\|_{W^{3/2,2}_{\vec\alpha}(\Gamma)} \right).
    \end{equation*}
    Moreover, with stability properties of $S_h$ and $Q_h$ we get
    \begin{equation*}
      \|S_hQ_h(z-z_h)\|_{H^1(\Omega)} \le c\, \|Q_h(z-z_h)\|_{H^{1/2}(\Gamma)}
      \le c\, \|z-z_h\|_{H^{1/2}(\Gamma)}
    \end{equation*}
    and with \eqref{eq:main_result_arbitrary} we conclude the assertion.
  \end{proof}

\section{Numerical experiments}\label{sec:experiments}

In order to confirm the theoretically predicted convergence results
we present some numerical experiments measuring the convergence rates.
Thus, we computed the problem
\eqref{eq:target}--\eqref{eq:state_eq} in the domains
\begin{align*}
  \Omega_{90} &= (0,1)^2,\\
  \Omega_{135} &= (-1,1)^2 \cap \{(r\cos\varphi,r\sin\varphi)\colon r\in (0,\infty), \varphi\in (0,3\pi/4)\},\\
  \Omega_{270} &= (-1,1)^2\setminus [0,1]^2,
\end{align*}
with input data $\nu = 1$, $f\equiv 0$ and $u_d(x_1,x_2) = x_1+x_2$.

We start with a structured grid consisting of 2, 3 or 6 triangles, respectively, and 
compute the discrete solutions solving \eqref{eq:optimality_system_disc}
on a sequence of meshes obtained by bisection of each element so that new nodes 
of the grid are inserted at the midpoints of the longest edge of each element.
The solution was computed by a GMRES method applied to the system \eqref{eq:variational_form_discrete} and in each iteration the linearized state and adjoint equation have to be solved. This
was done by the parallel direct solver MUMPS which allows to reuse the factorization of the
stiffness matrix. The implementation is written in C++ and the tests were performed on a
Intel-Core-i7-4770 (4x 3400MHz) machine with 32GB RAM.

As an explicit representation of the exact solution is not available for the given input data
we measured the error by comparison with the solution on a very fine mesh
with maximal element diameter $h_{ref}=2^{-10}$.
In Tables \ref{tab:domain_90}, \ref{tab:domain_135} and \ref{tab:domain_270} we report the error of the state in $H^1(\Omega)$, and the control in the 
$L^2(\Gamma)$-norm and the $H^{1/2}(\Gamma)$-seminorm, respectively. 
The latter norm is realized by the discrete harmonic extension $S_h$, this is, 
\[
  |z-z_h|_{H^{1/2}(\Gamma)} \approx |z_{h_{ref}} - z_h|_{H^{1/2}(\Gamma)} \sim \|\nabla S_{h_{ref}} (z_{h_{ref}} - z_h)\|_{L^2(\Omega)}.
\]

The convergence rates measured for the domain $\Omega_{90}$ are the same as in the experiments
from \cite{OPS13}. These results confirm the rates predicted in Theorem~\ref{thm:ocp_h12_error}
and Corollary~\ref{cor:estimate_states}.
Note that the largest opening angle is $\bar\omega=\pi/2$ and thus, $\bar\lambda = 2$.
Our theory moreover claims that the convergence rate for the discrete control is reduced when
the largest opening angle
exceeds the limiting case $2\pi/3$.
This is the case for the domain $\Omega_{135}$, where we have 
$\bar\omega = 3\pi/4$ and $\bar\lambda = 4/3$. The rate $4/3$ for the the control in the
$H^{1/2}(\Gamma)$-norm claimed in Theorem~\ref{thm:ocp_h12_error} is the rate we also
observe numerically. The convergence rate for the discrete states is still $1$ as
proved in Corollary~\ref{cor:estimate_states}.
The fact that our error estimates are also valid and sharp for non-convex domains is confirmed
by the experiment for the domain $\Omega_{270}$. Here, the rate $\bar\lambda=2/3$ is almost observed numerically for the discrete states and controls in $H^1(\Omega)$ and $H^{1/2}(\Gamma)$, respectively.
Note that the convergence rates in the experiments are always slightly better than predicted which is due to the approximate computation of the error by comparison with a reference solution on a fine grid.

Moreover, we have to notice that we have not proved error estimates for the control in $L^2(\Gamma)$,
but the experiments confirm in all cases that this convergence rate is higher by $1/2$
compared to the rate obtained in the $H^{1/2}(\Gamma)$-norm.
In order to obtain estimates in $L^2(\Gamma)$ one has to establish a Nitsche trick for
the non-conforming approximation \eqref{eq:variational_form_discrete} of \eqref{eq:variational_form}.
This will be subject of future research.

\begin{table}[htbp]
  \centering
  \begin{tabular}{rrrrrrr}
    \toprule
     \multicolumn{1}{l}{$h\cdot\sqrt{2}$} & \multicolumn{1}{l}{$\#BdDof$} & \multicolumn{1}{l}{$\#Dof$} & \multicolumn{1}{l}{$|u-u_h|_{H^1(\Omega)}$} & \multicolumn{1}{l}{$\|z-z_h\|_{L^2(\Gamma)}$} & \multicolumn{1}{l}{$|z-z_h|_{H^{1/2}(\Gamma)}$} \\
    \midrule
    % $2^{-1}$ & 9 & 16 & 2.11e-02 (0.00) & 2.84e-03 (0.00) & 1.07e-02 (0.00) \\ 
    % $2^{-2}$ & 49 & 32 & 1.25e-02 (0.75) & 9.99e-04 (1.51) & 4.24e-03 (1.33) \\ 
    % $2^{-3}$ & 225 & 64 & 6.63e-03 (0.92) & 2.50e-04 (2.00) & 1.44e-03 (1.56) \\ 
    $2^{-4}$ & 961 & 128 & 3.37e-03 (0.98) & 6.27e-05 (2.00) & 4.93e-04 (1.55) \\ 
    $2^{-5}$ & 3969 & 256 & 1.69e-03 (1.00) & 1.57e-05 (2.00) & 1.71e-04 (1.53) \\ 
    $2^{-6}$ & 16129 & 512 & 8.43e-04 (1.00) & 3.92e-06 (2.00) & 6.04e-05 (1.50) \\ 
    $2^{-7}$ & 65025 & 1024 & 4.19e-04 (1.01) & 9.72e-07 (2.01) & 2.22e-05 (1.44) \\ 
    $2^{-8}$ & 261121 & 2048 & 2.04e-04 (1.04) & 2.35e-07 (2.05) & 8.64e-06 (1.36) \\ 
    $2^{-9}$ & 1046530 & 4096 & 9.13e-05 (1.16) & 5.07e-08 (2.22) & 3.30e-06 (1.39) \\     
    \midrule
    \multicolumn{3}{l}{Theory:} & \phantom{0.00e+00} (1.00) &  \phantom{0.00e+00} (2.00) &  \phantom{0.00e+00} (1.50)\\
    \bottomrule
\end{tabular}
\caption{Results of the numerical experiment for the domain $\Omega_{90}$ showing finite
  element error and corresponding experimental convergence rates (in parentheses) for the state and
control.}
\label{tab:domain_90}
\end{table}

\begin{table}[htbp]
  \centering
\begin{tabular}{rrrrrr}
\toprule
  \multicolumn{1}{l}{$h\cdot\sqrt2$} & \multicolumn{1}{l}{$\#BdDof$} & \multicolumn{1}{l}{$\#Dof$} & \multicolumn{1}{l}{$|u-u_h|_{H^1(\Omega)}$} & \multicolumn{1}{l}{$\|z-z_h\|_{L^2(\Gamma)}$} & \multicolumn{1}{l}{$|z-z_h|_{H^{1/2}(\Gamma)}$} \\
  \midrule
%   $2^{-1}$ & 15 & 20 & 4.41e-02 (0.00) & 8.27e-03 (0.00) & 2.00e-02 (0.00) \\ 
% $2^{-2}$ & 77 & 40 & 2.27e-02 (0.96) & 2.17e-03 (1.93) & 7.45e-03 (1.43) \\ 
% $2^{-3}$ & 345 & 80 & 1.17e-02 (0.95) & 5.79e-04 (1.90) & 2.79e-03 (1.42) \\ 
$2^{-4}$ & 1457 & 160 & 5.93e-03 (0.98) & 1.56e-04 (1.90) & 1.06e-03 (1.40) \\ 
$2^{-5}$ & 5985 & 320 & 2.98e-03 (0.99) & 4.19e-05 (1.89) & 4.05e-04 (1.38) \\ 
$2^{-6}$ & 24257 & 640 & 1.49e-03 (1.00) & 1.13e-05 (1.89) & 1.58e-04 (1.36) \\ 
$2^{-7}$ & 97665 & 1280 & 7.44e-04 (1.01) & 3.05e-06 (1.89) & 6.22e-05 (1.34) \\ 
$2^{-8}$ & 391937 & 2560 & 3.63e-04 (1.03) & 7.98e-07 (1.93) & 2.48e-05 (1.33) \\ 
$2^{-9}$ & 1570300 & 5120 & 1.63e-04 (1.16) & 1.80e-07 (2.15) & 9.35e-06 (1.41) \\ 
  \midrule
  \multicolumn{3}{l}{Theory:} & \phantom{0.00e+00} (1.00) &  \phantom{0.00e+00} (1.83) &  \phantom{0.00e+00} (1.33)\\
    \bottomrule
\end{tabular}
\caption{Results of the numerical experiment for the domain $\Omega_{135}$ showing finite
  element error and corresponding experimental convergence rates (in parentheses) for the state and
control.}
\label{tab:domain_135}
\end{table}

\begin{table}[htbp]
  \centering
\begin{tabular}{rrrrrr}
\toprule
  \multicolumn{1}{l}{$h\cdot\sqrt2$} & \multicolumn{1}{l}{$\#BdDof$} & \multicolumn{1}{l}{$\#Dof$} & \multicolumn{1}{l}{$|u-u_h|_{H^1(\Omega)}$} & \multicolumn{1}{l}{$\|z-z_h\|_{L^2(\Gamma)}$} & \multicolumn{1}{l}{$|z-z_h|_{H^{1/2}(\Gamma)}$} \\
  \midrule
  % $2^{-1}$ & 33 & 32 & 7.36e-01 (0.00) & 4.76e-01 (0.00) & 4.78e-01 (0.00) \\ 
  % $2^{-2}$ & 161 & 64 & 4.52e-01 (0.70) & 2.17e-01 (1.14) & 2.73e-01 (0.81) \\ 
  % $2^{-3}$ & 705 & 128 & 2.77e-01 (0.71) & 8.97e-02 (1.27) & 1.61e-01 (0.76) \\ 
  $2^{-4}$ & 2945 & 256 & 1.71e-01 (0.70) & 3.68e-02 (1.28) & 9.83e-02 (0.71) \\ 
  $2^{-5}$ & 12033 & 512 & 1.06e-01 (0.69) & 1.51e-02 (1.29) & 6.07e-02 (0.69) \\ 
  $2^{-6}$ & 48641 & 1024 & 6.53e-02 (0.69) & 6.11e-03 (1.30) & 3.74e-02 (0.70) \\ 
  $2^{-7}$ & 195585 & 2048 & 4.01e-02 (0.71) & 2.43e-03 (1.33) & 2.25e-02 (0.73) \\ 
  $2^{-8}$ & 784385 & 4096 & 2.38e-02 (0.75) & 9.03e-04 (1.43) & 1.27e-02 (0.82) \\ 
  $2^{-9}$ & 3141630 & 8192 & 1.27e-02 (0.91) & 2.68e-04 (1.75) & 6.05e-03 (1.07) \\ 
  \midrule
      \multicolumn{3}{l}{Theory:} & \phantom{0.00e+00} (0.67) &  \phantom{0.00e+00} (1.17) &  \phantom{0.00e+00} (0.67)\\
    \bottomrule
\end{tabular}
\caption{Results of the numerical experiment for the domain $\Omega_{270}$ showing finite
  element error and corresponding experimental convergence rates (in parentheses) for the state and
control.}
\label{tab:domain_270}
\end{table}
\bigskip

\textbf{Acknowledgement:}
  The author acknowledges the fruitful discussions with Johannes Pfefferer
  during the preparation of the manuscript.

\bibliographystyle{plain}
\bibliography{bibliography2}
\end{document}